\RequirePackage[l2tabu,orthodox]{nag}
\documentclass[a4paper,11pt]{amsart}
\usepackage{amsmath,amsfonts,amsthm,amssymb,amsthm,stmaryrd}
\usepackage[all]{xy}
\usepackage[margin=1.25in]{geometry}

\usepackage[T1]{fontenc}
\usepackage[sc]{mathpazo}

\usepackage{hyperref}
\hypersetup{
  colorlinks=true,
  allcolors=blue
}

\usepackage{mathrsfs}

\author{Yue Fan}
\title{Construction of the Moduli Space of Higgs Bundles
  Using Analytic Methods}
\date{\today}
\address{Yue Fan, Department of Mathematics, University of Maryland, College
  Park, MD 20742, USA}
\email{\texttt{yuefan@umd.edu}}

\renewcommand{\bar}[1]{\overline{#1}}
\renewcommand{\tilde}[1]{\widetilde{#1}}

\newcommand{\sC}{\mathscr{C}}
\newcommand{\sO}{\mathscr{O}}

\newcommand{\G}{\mathscr{G}}
\newcommand{\A}{\mathscr{A}}
\newcommand{\B}{\mathscr{B}}
\newcommand{\M}{\mathscr{M}}
\newcommand{\U}{\mathscr{U}}
\newcommand{\V}{\mathscr{V}}
\newcommand{\sL}{\mathscr{L}}
\newcommand{\Z}{\mathscr{Z}}
\newcommand{\K}{\mathscr{K}}
\newcommand{\F}{\mathscr{F}}
\newcommand{\E}{\mathscr{E}}
\newcommand{\g}{\mathfrak{g}}

\newcommand{\C}{\mathbb{C}}
\newcommand{\R}{\mathbb{R}}
\newcommand{\bH}{\mathbf{H}}

\newtheorem{proposition}{Proposition}[section]
\newtheorem{lemma}[proposition]{Lemma}

\newtheorem{theorem}[proposition]{Theorem}
\newtheorem{theoremintro}{Theorem}

\newtheorem{corollary}[proposition]{Corollary}

\numberwithin{equation}{section}

\DeclareMathOperator{\im}{im}
\DeclareMathOperator{\Aut}{Aut}

\DeclareMathOperator{\End}{End}
\DeclareMathOperator{\grad}{grad}

\begin{document}
\begin{abstract}
  It is a folklore theorem that the Kuranishi slice method can be used
  to construct the moduli space of semistable Higgs bundles on a
  closed Riemann surface as a complex space. The purpose of this paper
  is to provide a proof in detail. We also give a direct proof that
  the moduli space is locally modeled on an affine GIT quotient of a
  quadratic cone by a complex reductive group.
\end{abstract}
\maketitle
\tableofcontents

\section{Introduction}
Let $X$ be a closed Riemann surface with genus $g\geq2$. Introduced by
Hitchin in the seminal paper \cite{Hitchin1987b}, a Higgs bundle on
$X$ is a pair $(\E,\Phi)$ consisting of a holomorphic bundle $\E\to X$ and a
holomorphic section $\Phi\in H^0(\End\E\otimes\K_X)$, where $\K_X$ is the
canonical bundle of $X$. To obtain a nice moduli space, we recall that
a Higgs bundle $(\E,\Phi)$ is stable if $\mu(\F)<\mu(\E)$ for every
$\Phi$-invariant holomorphic subbundle $0\subsetneq\F\subsetneq\E$,
where $\mu(\F)$ is the slope of $\F$. The semistability is defined by
replacing $\mu(\F)<\mu(\E)$ by $\mu(\F)\leq\mu(\E)$. Finally,
$(\E,\Phi)$ is polystable if it is a direct sum of stable Higgs
bundles with the same slope. In \cite{Hitchin1987b}, Hitchin used the
Kuranishi slice method to construct the moduli space of stable Higgs
bundles first as a smooth manifold and then as a hyperK\"ahler
manifold. Such a method was first introduced by Kuranishi in
\cite{kuranishi1965new} and has been used in several papers to
construct moduli spaces in different contexts (for example, see
\cite{atiyah1978self,lubke1987moduli,Bradlow1991} and \cite[Chapter
7]{Kobayashi2014}). On the other hand, the moduli space of semistable
Higgs bundles was constructed by Nitsure in \cite{nitsure1991moduli}
where $X$ is a smooth projective curve and by Simpson in
\cite{Simpson1994} where $X$ is a smooth projective variety. They both
used Geometric Invariant Theory (GIT for short), and the method is
entirely algebro-geometric. As a consequence, the resulting moduli
space is a quasi-projective variety.

It is a folklore theorem that the Kuranishi slice method can be used
to construct the moduli space of semistable Higgs bundles as a complex
space (for example, see
\cite{wentworth2016higgs,bradlow2012morse}). The purpose of this paper
is to provide a proof in detail. More precisely, the problem is stated
as follows. Fix a smooth Hermitian vector bundle $E\to X$ and let
$\g_E\to X$ be the bundle of skew-Hermitian endomorphisms of $E$. For
convenience, we assume that the degree of $E$ is zero. This condition
is not essential. By the Newlander-Nirenberg theorem, a holomorphic
structure on $E$ (described by holomorphic transition functions) is
equivalent to an integrable Dolbeault operator
$\bar{\partial}_E$. Since $\dim_\C X=1$, the integrability condition
is vacuous. Therefore, via the Chern correspondence, the space of
holomorphic structures on $E$ can be identified with the space $\A$ of
unitary connections on $E$, which is an infinite-dimensional affine
space modeled on $\Omega^1(\g_E)$. Let
$\sC=\A\times\Omega^{1,0}(\g_E^\C)$. Then, the configuration space of
Higgs bundles (with a fixed underlying smooth bundle $E$) is defined
as
\begin{equation*}
  \B=\{(A,\Phi)\in\sC\colon\bar{\partial}_A\Phi=0\}
\end{equation*}
(see \cite{wentworth2016higgs} for more details). Since the complex
gauge group $\G^\C=\Aut(E)$ naturally acts on the space of holomorphic
structures of $E$, it acts on $\A$ and hence also on $\sC$. Then, two
Higgs bundles are isomorphic if and only if they are in the same
$\G^\C$-orbit. Let $\B^{ss}$, $\B^s$ and $\B^{ps}$ be the subspaces of
$\B$ consisting of semistable, stable and polystable Higgs bundles,
respectively. They are $\G^\C$-invariant. The moduli space of
semistable Higgs bundles is defined as the quotient $\M=\B^{ps}/\G^\C$
equipped with the $C^\infty$-topology. Our main result is the
following.
\begin{theoremintro}\label{sec:introduction-is-a-complex-space}
  The moduli space $\M$ is a normal complex space.
\end{theoremintro}
More can be said about the local structure of $\M$. To state the
theorem, we need some preparation. Recall that the space $\sC$ has a
natural $L^2$-metric $g$ and a compatible complex structure $I$ given
by multiplication by $i$ (see \cite[\S6]{Hitchin1987b}). Let $\G$ be
the subgroup of $\G^\C$ consisting of unitary gauge
transformations. Then, the $\G$-action on $\sC$ is Hamiltonian with
respect to the K\"ahler form $\Omega_I=g(I\cdot,\cdot)$. Hitchin's
equation can be interpreted as a moment map
\begin{equation}\label{eq:HitchinEq}
  \mu(A,\Phi)=F_A+[\Phi,\Phi^*]
\end{equation}
Then, the Hitchin-Kobayashi correspondence (see
\cite{Hitchin1987b,Simpson1988}) states that a Higgs bundle is
polystable if and only if its $\G^\C$-orbit intersects
$\mu^{-1}(0)$. Moreover, the inclusion
$\mu^{-1}(0)\cap\B\hookrightarrow\B^{ps}$ induces a homeomorphism
\begin{equation*}
     (\mu^{-1}(0)\cap\B)/\G\xrightarrow{\sim}\B^{ps}/\G^\C
\end{equation*}
whose inverse is induced by the retraction
$r\colon\B^{ss}\to\mu^{-1}(0)$ defined by the Yang-Mills-Higgs flow
(see \cite{Wilkin2006}). Finally, we recall the deformation complex
for a Higgs bundle $(A,\Phi)$:
\begin{equation*}
  C_{\mu_\C}\colon\qquad\Omega^0(\g_E^\C)\xrightarrow{D''}\Omega^{0,1}(\g_E^\C)\oplus\Omega^{1,0}(\g_E^\C)\xrightarrow{D''}\Omega^{1,1}(\g_E^\C)
\end{equation*}
where $D''=\bar{\partial}_A+\Phi$. It is an elliptic complex. Let $K$
be the $\G$-stabilizer at $(A,\Phi)$. Since the $\G$-action is proper,
$K$ is a compact Lie group. Moreover, its complexification $K^\C$ is
precisely the $\G^\C$-stabilizer at $(A,\Phi)$ (see
Section~\ref{sec:kuran-local-models}) and acts on $\bH^1$
linearly. Then, the local structure of $\M$ is described as follows.
\begin{theoremintro}\label{sec:introduction-local-model}
  Let $[A,\Phi]\in\M$ be a point such that $\mu(A,\Phi)=0$ and $\bH^1$
  its deformation space, the harmonic space $\bH^1(C_{\mu_\C})$
  defined in $C_{\mu_\C}$. Then, the following hold:
  \begin{enumerate}
  \item $\bH^1$ is a complex symplectic vector space.
    
  \item The $K^\C$-action on $\bH^1$ is complex Hamilnotian with a
    complex moment map given by
    \begin{equation*}
      \nu_{0,\C}(x)=\frac{1}{2}H[x,x]
    \end{equation*}
    where $H$ is the harmonic projection defined in $C_{\mu_\C}$.
    
  \item Around $[A,\Phi]$, the moduli space $\M$ is locally
    biholomorphic to an open neighborhood of $[0]$ in the complex
    symplectic quotient $\nu_{0,\C}^{-1}(0)\sslash K^\C$, which is an
    affine GIT quotient.
  \end{enumerate}
\end{theoremintro}

There are two reasons why this result is not surprising. In
\cite[\S10]{Simpson1994}, Simpson proved that the differential graded
Lie algebra $C_{\mu_\C}$ is \textit{formal}. As a consequence, the
moduli space is locally biholomorphic to a GIT quotient of a quadratic
cone in $H^1(C_{\mu_\C})$ by a complex reductive group. Another reason
is the following. Recall that $\sC$ is more than just a K\"ahler
manifold. It has a hyperK\"ahler structure (see
\cite[\S6]{Hitchin1987b}) and admits a complex moment map
$\mu_\C(A,\Phi)=\bar{\partial}_A\Phi$ for the $\G^\C$-action. Hence,
the moduli space $\M$ is homeomorphic, by the Hitchin-Kobayashi
correspondence, to a singular hyperK\"ahler quotient. Then,
Theorem~\ref{sec:introduction-local-model} is an infinite-dimensional
generalization of Theorem 1.4(iv) in Mayrand \cite{mayrand2018local}
to Higgs bundles. We will extend all other statements in Theorem 1.4
to $\M$ in a forthcoming paper.

The major step in the proof of
Theorem~\ref{sec:introduction-is-a-complex-space} and
\ref{sec:introduction-local-model} is to construct a Kuranishi local
model for $\M$ at every Higgs bundle $(A,\Phi)$ that satisfies
Hitchin's equation. This is done in
Section~\ref{sec:kuran-local-models}. Here, a Kuranishi local model is
the analytic GIT quotient (developed by Heinzner and Loose in
\cite{heinzner1994reduction}) of a Kuranishi space in $\bH^1$ by the
$\G^\C$-stabilizer at $(A,\Phi)$, and is homeomorphic to an open
neighborhood of $(A,\Phi)$ in $\M$. After that, we will show that the
transition functions associated with Kuranishi local models are
holomorphic so that $\M$ is a complex space. This is done in
Section~\ref{sec:gluing-local-models}. To prove
Theorem~\ref{sec:introduction-local-model}, we adapt Huebschmann's
argument in \cite[Corollary 2.20]{huebschmann1995singularities} which
is further based on Arms-Marsden-Moncrief
\cite{arms1981symmetry}. This is done in
Section~\ref{sec:sing-kuran-spac}.

The techniques in the construction of Kuranishi local models mainly
come from \cite{szekelyhidi2010kahler}, \cite{dervan2018moduli} and
\cite{inoue2019moduli}. Let $K$ be the $\G$-stabilizer at $(A,\Phi)$
with $\mu(A,\Phi)=0$ so that $K^\C$ is the $\G^\C$-stabilizer. We will
construct a $K$-equivariant perturbed Kuranishi map $\Theta$
(following Sz\'ekelyhidi's argument in \cite[Proposition
7]{szekelyhidi2010kahler}) that is defined on a Kuranishi space in
$\bH^1$ and takes values in $\B^{ss}$ such that the pullback moment
map $\Theta^*\mu$ is a moment map for the $K$-action on $\bH^1$ with
respect to the pullback symplectic form $\Theta^*\Omega_I$. Then,
roughly speaking, a $K^\C$-orbit is closed in $\bH^1$ if and only if
it contains a zero of the pullback moment map $\Theta^*\mu$. The
precise statement is given in
Theorem~\ref{sec:kuran-local-models-closedness-polystability}
(cf. \cite[Theorem 2.9]{dervan2018moduli}, \cite[Proposition
3.8]{inoue2019moduli}, \cite[Proposition 2.4]{bronnle2011deformation}
and \cite[Proposition 3.3.2]{van2015deformations}). Since the
perturbed Kuranishi map $\Theta$ is no longer holomorphic,
$\Theta^*\Omega_I$ is not a K\"ahler form on $\bH^1$, which causes
some trouble. To remedy this problem, in the proof of
Theorem~\ref{sec:kuran-local-models-closedness-polystability}, the
Yang-Mills-Higgs flow will be used to detect polystable orbits in
$\B^{ss}$. Since Kuranishi spaces are locally complete, every
Yang-Mills-Higgs flow near $(A,\Phi)$ induce a ``reduced flow'' in
$\bH^1$ that stays in a single $K^\C$-orbit and converges to a zero of
$\Theta$. Therefore, if a $K^\C$-orbit is closed, it contains a zero
of $\Theta$. Hence, $\Theta$ maps polystable $K^\C$-orbits in $\bH^1$
to polystable orbits in $\B^{ss}$ so that $\Theta$ induces a map from
a Kuranishi local model to $\M$. The rest of the proof is to show that
this map is an open embedding.

After the construction of the moduli space $\M$, it is natural to
compare the analytic and the algebraic moduli spaces. More precisely,
let us also use $\M_{an}$ to mean the quotient $\B^{ps}/\G^\C$ and
$\M_{alg}$ the moduli space of semistable Higgs bundles of rank $r$
and degree $0$ in the category of schemes, where $r$ is the rank of
$E$. By construction, $\M_{alg}$ parametrizes s-equivalence classes of
Higgs bundles. Let us recall the definition of s-equivalence. Every
semistable Higgs bundle $(\E,\Phi)$ admits a filtration, called the
Seshadri filtration, whose successive quotients are stable, all with
slope $\mu(E)$. Let $\mathrm{Gr}(\E,\Phi)$ be the graded object
associated with the Seshadri filtration of $(\E,\Phi)$. It is uniquely
determined by the isomorphism class of $(\E,\Phi)$. Then, two Higgs
bundles $(\E_1,\Phi_1)$ and $(\E_2,\Phi_2)$ are s-equivalent if
$\mathrm{Gr}(\E_1,\Phi_1)$ and $\mathrm{Gr}(\E_2,\Phi_2)$ are
isomorphic as Higgs bundles. As a consequence, there is a natural
comparison map $i\colon\M_{an}\to\M_{alg}$ of the underlying sets that
sends each $\G^\C$-orbit of a point $(A,\Phi)$ in $\B^{ps}$ to the
s-equivalence class of the Higgs bundle $(\E_A,\Phi)$ defined by
$(A,\Phi)$. The following result will be proved in
Section~\ref{sec:comp-betw-analyt}.
\begin{theoremintro}\label{sec:introduction-comparison}
  The comparison map $i\colon\M_{an}\to\M_{alg}$ is a biholomorphism
\end{theoremintro}
The outline of the proof is the following. It is easy to see that $i$
is a bijection. To show that it is continuous, recall that Nitsure
constructed a scheme $F^{ss}$ in \cite{nitsure1991moduli} that
parameterizes semistable Higgs bundles on $X$, and $\M_{alg}$ is a
good quotient of $F^{ss}$. We show that the comparison map $i$ can be
locally lifted to a map $\sigma$, called a \textit{classifying map},
that is defined locally on $\B^{ss}$ and takes values in
$F^{ss}$. Here, the terminology comes from Sibley and Wentworth's
paper \cite{sibley2019continuity}, and we adapt the proof of Theorem
6.1 in this paper to show that $\sigma$ is continuous with respect to
the $C^\infty$-topology on $\B^{ss}$ and the analytic topology on
$F^{ss}$. Therefore, $i$ is continuous. By the properness of the
Hitchin fibration defined on $\M_{an}$, we see that $i$ is proper and
hence a homeomorphism. Then, by constructing Kuranishi families of
stable Higgs bundles, we show that the restriction
$i\colon\M_{an}^s\to\M_{alg}^s$ is a biholomorphism, where $\M_{an}^s$
and $\M_{alg}^s$ are the open subsets of $\M_{an}$ and $\M_{alg}$
consisting of stable Higgs bundles, respectively. Then, we use
Theorem~\ref{sec:introduction-local-model} to prove that $\M_{an}$ is
normal. Since $\M_{alg}\setminus\M_{alg}^s$ has codimension $\geq2$,
by the normality and the reducedness of $\M_{an}$ and $\M_{alg}$, the
holomorphicity of $i^{-1}|_{\M_{alg}^s}$ can be extended to
$i^{-1}$. The rest of the proof follows from the fact that a
holomorphic bijection between normal, reduced and irreducible complex
spaces of the same dimension is a biholomorphism.

After this paper was complete, we were aware of Buchdahl and
Schumacher's paper \cite{buchdahl2020polystability}. Note that
Theorem~\ref{sec:kuran-local-models-closedness-polystability} is
similar to \cite[Theorem 3]{buchdahl2020polystability}, which applies
to holomorphic vector bundles over a compact K\"ahler
manifold. However, our approaches are different. In this paper, the
Yang-Mills-Higgs flow plays a major role. Since $\dim_\C X=1$, the
necessary analytic inputs are from Wilkin \cite{Wilkin2006}. By
contrast, the Yang-Mills flow is not involved in Buchdahl and
Schumacher's argument. It is expected that Buchdahl and Schumacher's
argument can be adapted to the case of Higgs bundles and used to
provide another proof of
Theorem~\ref{sec:introduction-is-a-complex-space}, possibly without
the assumption that $\dim_\C X=1$.

Finally, we remark that we only work with reduced complex spaces in
this paper. The reason is that the analytic GIT developed by Heinzner
and Loose in \cite{heinzner1994reduction} only apply to reduced
complex spaces.

\vspace{0.5cm}
\noindent\textbf{Acknowledgments}. This paper is part of my Ph.D. thesis. I
would like to thank my advisor, Professor Richard Wentworth, for
suggesting this problem and his generous support and guidance. I also
would like to thank Maxence Mayrand, Johannes Huebschmann, Craig van
Coevering and Eiji Inoue for helpful discussions.

\section{Deformation complexes}
In this section, after reviewing the deformation complex for Higgs
bundles, we introduce another useful Fredholm complex that will be
used later. Let $(A,\Phi)\in\B$ such that $\mu(A,\Phi)=0$. Then,
consider the deformation complex
\begin{gather*}
  C_{\mu_\C}\colon\qquad
  \Omega^0(\g_E^\C)\xrightarrow{D''}\Omega^{0,1}(\g_E^\C)\oplus\Omega^{1,0}(\g_E^\C)\xrightarrow{D''}\Omega^{1,1}(\g_E^\C)
\end{gather*}
where $D''=\bar{\partial}_A+\Phi$. Recall that $C_{\mu_\C}$ is obtained
by linearizing the equation $\bar{\partial}_A\Phi=0$ and the
$\G^\C$-action.
\begin{proposition}[{\cite[\S1]{Simpson1992} and \cite[\S10]{Simpson1994}}]
  $C_{\mu_\C}$ is an elliptic complex and a differential graded Lie
  algebra. Moreover, the K\"ahler's identities
  \begin{equation*}
    (D'')^*=-i[*,D']\qquad (D')^*=+i[*,D'']
  \end{equation*}
  hold, where $D'=\partial_A+\Phi^*$ and $*$ is the Hodge star.
\end{proposition}
There is another useful sequence
\begin{equation*}
  C_{\mu}\colon\qquad \Omega^0(\g_E)\xrightarrow{d_1}\ker D''\xrightarrow{d_2}\Omega^2(\g_E)
\end{equation*}
where $d_2$ is the derivative of $\mu$~\eqref{eq:HitchinEq} at
$(A,\Phi)$, and $d_1(u)=(d_Au,[\Phi,u])$. The operator $d_2$, viewed
as a map $\Omega^1(\g_E)\oplus\Omega^{1,0}(\g_E^\C)\to\Omega^2(\g_E)$,
has a surjective symbol. Hence,
$d_2d_2^*\colon\Omega^2(\g_E)\to\Omega^2(\g_E)$ is a self-adjoint
elliptic operator. As a consequence, the Hodge decomposition
\begin{equation*}
  \Omega^2(\g_E^\C)=\im d_2d_2^*\oplus\ker d_2d_2^*,
\end{equation*}
holds. Moreover, since $d_2(D'')^*=0$ and
\begin{equation*}
  \Omega^{0,1}(\g_E^\C)\oplus\Omega^{1,0}(\g_E^\C)=\ker D''\oplus\im (D'')^*,
\end{equation*}
we have
\begin{equation*}
  d_2(\ker D'')=d_2(\Omega^1(\g_E)\oplus\Omega^{1,0}(\g_E^\C))
\end{equation*}
(In this paper, we routinely identify $\Omega^1(\g_E)$ with
$\Omega^{0,1}(\g_E^\C)$ using the map $\alpha\mapsto\alpha''$, where
$\alpha''$ is the $(0,1)$-component of $\alpha$). As a consequence,
the natural map $\ker d_2^*\to H^2(C_\mu)$ is an isomorphism. We
denote $\ker d_2^*$ by $\bH^2(C_\mu)$. Finally, we note that
$H^1(C_\mu)$ is equal to the first cohomology of the following
elliptic complex that is used by Hitchin in \cite[p. 85]{Hitchin1987b}.
\begin{equation*}
  C_{Hit}\colon\qquad
  \Omega^0(\g_E)\xrightarrow{d_1}\Omega^1(\g_E)\oplus\Omega^{1,0}(\g_E^\C)\xrightarrow{d_2\oplus
  D''}\Omega^2(\g_E)\oplus\Omega^{1,1}(\g_E^\C)
\end{equation*}
In fact, by direct computation, the identification
$\Omega^1(\g_E)\xrightarrow{\sim}\Omega^{0,1}(\g_E^\C)$ induces an
isomorphism
$\bH^1(C_{Hit})\xrightarrow{\sim}\bH^1(C_{\mu_\C})$. Therefore, in the
rest of the paper, if no confusion can appear, we will simply use
$\bH^1$ to mean the harmonic space $\bH^1(C_{\mu_\C})$. In summary, we
have obtained
\begin{proposition}
  The sequence $C_\mu$ is a Fredholm complex with Hodge decomposition
  \begin{equation*}
    \Omega^2(\g_E)=\bH^2(C_\mu)\oplus\im d_2
  \end{equation*}
\end{proposition}
Lastly, note that the natural non-degenerate pairing
$\Omega^0(\g_E)\times\Omega^2(\g_E)\to\R$ restricts to a
non-degenerate pairing $\bH^0(C_\mu)\times\bH^2(C_\mu)\to\R$ so that
$\bH^2(C_\mu)$ can be identified with the dual space $\bH^0(C_\mu)^*$
of $\bH^0(C_\mu)$.

\section{Kuranishi local models}\label{sec:kuran-local-models}

\subsection{Kuranishi maps}
A crucial ingredient in the Kuranishi slice method is the Kuranishi
maps. They relate polystable orbits in $\bH^1$ and polystable orbits
in $\B$. Moreover, they eventually induce local charts for the moduli
space. To construct Kuranishi maps, we need to use the implicit
function theorem, and it is a standard practice to work with the
Sobolev completions of relevant spaces. In this paper, we will use
$Y_k$ to mean the completion of the space $Y$ with respect to the
Sobolev $L_k^2$-norm. For example, $\Omega^*(\g_E)_k$ means the
completion of $\Omega^*(\g_E)$ with respect to the
$L_k^2$-norm. Otherwise, we generally use $C^\infty$-topology. Fix
$k>1$.

Now, we describe the Kuranishi maps. Let $(A,\Phi)\in\B$ with
$\mu(A,\Phi)=0$. Recall that $\G_{k+1}^\C$ and $\G_{k+1}$ are Hilbert
Lie groups and act smoothly on the Hilbert affine manifold
$\sC_k$. Moreover, the $\G_{k+1}$-action on $\sC_k$ is proper (see
\cite[Section 4.4]{friedman1998gauge}). Therefore, if $K$ is the
$\G_{k+1}$-stabilizer at $(A,\Phi)$, then $K$ is a compact Lie group
with Lie algebra $\bH^0(C_\mu)$. The following result relates the
$\G_{k+1}^\C$-stabilizer to the $\G_{k+1}$-stabilizer at $(A,\Phi)$.
\begin{proposition}\label{sec:kuran-local-models-stabilizers}
  The $\G_{k+1}^\C$-stabilizer at $(A,\Phi)$ is the complexification of
  $K$ and acts on $\bH^1$.
\end{proposition}
\begin{proof}
  This follows from \cite[Proposition
  1.6]{sjamaar1995holomorphic}. The rest follows from direct
  computation.
\end{proof}

If $\bH^2(C_{\mu_\C})=0$,
then the implicit function theorem implies that $\B_k$ is locally a
complex manifold around $(A,\Phi)$. In general, following
Lyapunov-Schmidt reduction, we consider
\begin{equation*}
  \tilde{\B}_k=[(1-H)\mu_\C]^{-1}(0)\subset\sC_k
\end{equation*}
where $H$ is the harmonic projection defined in the elliptic complex
$C_{\mu_\C}$. By construction, the derivative of $(1-H)\mu_\C$ at
$(A,\Phi)$ is surjective. Hence, $\tilde{\B}_k$ is locally a complex
manifold around $(A,\Phi)$. To parametrize $\tilde{\B}_k$,
consider the map
\begin{gather*}
  F\colon\Omega^{0,1}(\g_E^\C)_k\oplus\Omega^{1,0}(\g_E^\C)_k\to\Omega^{0,1}(\g_E^\C)_k\oplus\Omega^{1,0}(\g_E^\C)_k\\
  F(\alpha,\eta)=(\alpha,\eta)+(D'')^*G[\alpha'',\eta]
\end{gather*}
where $\alpha''$ is the $(0,1)$-component of $\alpha$.
It has the following properties.
\begin{lemma}\ \label{sec:kuran-local-models-Fmap-properties}
  \begin{enumerate}
  \item $F$ is $K^\C$-equivariant.
    
  \item $F$ is a local biholomorphism around $0$.
    
  \item $D''F(\alpha,\eta)=(1-H)\mu_\C(A+\alpha,\Phi+\eta)$.
    
  \item $(D'')^*F(\alpha,\eta)=(D'')^*(\alpha,\eta)$.
  \end{enumerate}
\end{lemma}
\begin{proof}
  $(1)$ follows from the fact that the $K^\C$-action commutes with
  $(D'')^*$ and $G$. Since the derivative of $F$ at $0$ is the identity
  map, the inverse function theorem implies $(2)$. Since
  $(D'')^*(D'')^*=0$, $(4)$ follows. To prove $(3)$, we compute
  \begin{align*}
    &(1-H)\mu_\C(A+\alpha,\Phi+\eta)\\
    &=D''(D'')^*G(D''(\alpha,\eta)+[\alpha'',\eta])\\
    &=D''((\alpha,\eta)-H(\alpha,\eta)
    -D''(D'')^*G(\alpha,\eta)+(D'')^*G[\alpha'',\eta])\\
    &=D''((\alpha,\eta)+(D'')^*G[\alpha,\eta])\\
    &=D''F(\alpha,\eta)
  \end{align*}
\end{proof}
As a consequence, $F$ induces a well-defined map,
\begin{equation*}
  F\colon\tilde{\B}_k\cap[(A,\Phi)+\ker (D'')^*]\to\ker D''\cap\ker (D'')^*=\bH^1
\end{equation*}
Since $\tilde{\B}_k$ and $(A,\Phi)+\ker (D'')^*$ intersect
transversely at $(A,\Phi)$, their intersection is locally a complex
manifold around $(A,\Phi)$. Hence, there are an open ball
$U\subset\bH^1$ in the $L^2$-norm around $0$ and an open neighborhood
$\tilde{U}$ of $(A,\Phi)$ in $\tilde{\B}_k\cap[(A,\Phi)+\ker (D'')^*]$
such that $F\colon \tilde{U}\to U$ is a biholomorphism. The
\textit{Kuranishi map} $\theta$ is defined as its inverse viewed as a
map $\theta\colon U\hookrightarrow\sC_k$, and the \textit{Kuranishi
  space} is defined as $Z:=\theta^{-1}(\B\cap\tilde{U})$. More
concretely, by the construction of $\tilde{\B}_k$,
\begin{equation*}
  Z:=\{x\in U\colon H[\theta(x),\theta(x)]=0\}
\end{equation*}
Here, $(A,\Phi)$ serves as the origin in the affine manifold
$\sC_k$. Clearly, $Z$ is a closed complex subspace of $U$. Moreover,
since $\B^{ss}_k$ is open in $\B_k$ (see \cite[Theorem
4.1]{Wilkin2006}), by shrinking $U$ and hence $Z$ if necessary, we may
assume that $\theta(Z)\subset\B^{ss}_k$.

The next result shows that the Kuranishi space $Z$ is locally
complete.
\begin{proposition}\label{sec:kuran-local-models-local-completeness}
  The map
  \begin{gather*}
    T\colon\bH^0(C_\mu)^\perp_{k+1}\times \bH^2(C_\mu)^\perp_{k+1}\times[((A,\Phi)+\ker
    (D'')^*)\cap\B^{ss}_k]\to\B^{ss}_k\\
    T(u,\beta,B,\Psi)=(B,\Psi)\cdot\exp(-i*\beta)\exp(u)
  \end{gather*}
  is a local homeomorphism around $(0,0,A,\Phi)$. As a consequence,
  there exists an open neighborhood $W$ of $(A,\Phi)$ in $\B^{ss}_k$
  such that the $\G_{k+1}^\C$-orbit of every $(B,\Psi)\in W$ intersects
  the image $\theta(Z)$.
\end{proposition}
\begin{proof}
  Consider the map
  \begin{gather*}
    T\colon\bH^0(C_\mu)_{k+1}^\perp\times \bH^2(C_\mu)_{k+1}^\perp\times((A,\Phi)+\ker
    (D'')^*)\to\sC_k\\
    T(u,\beta,B,\Psi)=(B,\Psi)\cdot\exp(-i*\beta)\exp(u)
  \end{gather*}
  where $\bH^0(C_\mu)^\perp$ and $\bH^2(C_\mu)^\perp$ are the
  $L^2$-orthogonal complements of $\bH^0(C_\mu)$ and $\bH^2(C_\mu)$ in
  $\Omega^0(\g_E)$ and $\Omega^2(\g_E)$, respectively. Its derivative
  at $(0,0,A,\Phi)$ is given by
  \begin{equation*}
    d_{(0,0,A,\Phi)}
    T(u,\beta,x)=D''(u-i*\beta)+x
  \end{equation*}
  Note that
  \begin{equation*}
    \bH^0(C_\mu)^\perp\oplus
    i*\bH^2(C_\mu)^\perp=\bH^0(C_\mu)^\perp\oplus i\bH^0(C_\mu)^\perp=\bH^0(C_{\mu_\C})^\perp
  \end{equation*}
  Since
  \begin{equation*}
    \Omega^{0,1}(\g_E^\C)_k\oplus\Omega^{1,0}(\g_E^\C)_k=\ker
    (D'')^*\oplus\im D''
  \end{equation*}
  we conclude that $d_{(0,A,\Phi)}T$ is an isomorphism. Hence, the
  inverse function theorem implies that there are open neighborhoods
  $N_1\times N_2\times V$ of $(0,0,A,\Phi)$ and $W$ of $(A,\Phi)$ such
  that $T\colon N_1\times N_2\times V\to W$ is a diffeomorphism. Since
  $\B^{ss}_k$ is $\G_{k+1}^\C$-invariant, we conclude that
  \begin{equation*}
    T\colon N_1\times N_2\times (V\cap\B^{ss}_k)\to W\cap\B^{ss}_k
  \end{equation*}
  is a homeomorphism. Finally, if $U\subset\bH^1$ is sufficiently
  small, then $\theta$ is a homeomorphism from $Z$ to $V\cap\B^{ss}_k$.
\end{proof}
Moreover, $\theta$ maps $K^\C$-orbits to $\G^\C$-orbits in the
following way.
\begin{proposition}[cf. {\cite[Lemma
    6.1]{chen2014calabi}}]\label{sec:kuran-local-models-infi-slice-theorem}\
  If $U$ is sufficiently small, then the following hold:
  \begin{enumerate}
  \item If $x_1,x_2\in U$ are such that $x_1=x_2g$ for some $g\in
    K^\C$, then $\theta(x_1)=\theta(x_2)g$. Hence, if $x_1\in Z$, then
    $x_2\in Z$.
    
  \item Conversely, if $d_x\theta(v)=u^\#_{\theta(x)}$ for some
    $u\in\Omega^0(\g_E^\C)_{k+1}$, then $u\in\bH^0(C_{\mu_\C})$, and
    $v=u^\#_x$, where $u^\#$ is the infinitesimal action of $u$.
  \end{enumerate}
\end{proposition}
\begin{proof}
  Since $U$ is an open ball around $0$, it is orbit-convex by
  \cite[Lemma 1.14]{sjamaar1995holomorphic}. Hence, the holomorphicity
  of $\theta$ and \cite[Proposition 1.4]{sjamaar1995holomorphic} imply
  that $\theta(x_1)=\theta(x_2)g$. Since $\B^{ss}_k$ is
  $\G_{k+1}^\C$-invariant, if $\theta(x_1)\in\B^{ss}_k$, then
  $\theta(x_2)\in\B^{ss}_k$ so that $x_2\in Z$. To prove $(2)$, we
  claim that $u\in\bH^0(C_{\mu_\C})$. Then, the claim implies that
  \begin{equation*}
    v=d_{\theta(x)}F(d_x\theta(v))=d_{\theta(x)}F(u^\#_{\theta(x)})=\frac{d}{dt}\biggr|_{t=0}F(\theta(x)e^{tu})=\frac{d}{dt}\biggr|_{t=0}xe^{tu}=u^\#_x
  \end{equation*}
  To prove the claim, write $u=u'+u''$ for some
  $u'\in\bH^0(C_{\mu_\C})$ and
  $u''\in\bH^0(C_{\mu_\C})^\perp_{k+1}$. Since $\theta$ takes values
  in $(A,\Phi)+\ker (D'')^*$, $(u'')^\#_{\theta(x)}\in\ker
  (D'')^*$. In the proof of
  Proposition~\ref{sec:kuran-local-models-local-completeness}, we see
  that the map
  \begin{equation*}
    T\colon\bH^0(C_\mu)^\perp_{k+1}\times\bH^2(C_{\mu_\C})^\perp_{k+1}\times((A,\Phi)+\ker
    (D'')^*)\to\sC_k
  \end{equation*}
  is a local diffeomorphism around $(0,0,A,\Phi)$. Hence, there are
  open neighborhoods $N_1\times N_2\times V$ of $(0,0,A,\Phi)$ and $W$
  of $(A,\Phi)$ such that $T\colon N_1\times N_2\times V\to W$ is a
  diffeomorphism. If $U$ is sufficiently small,
  $\theta\colon Z\to V\cap\B^{ss}_k$ is a homeomorphism. Therefore, the
  derivative $d_{(0,0,\theta(x))}T$ of $T$ is injective. Note that
  \begin{equation*}
    \bH^0(C_\mu)^\perp\oplus
    i*\bH^2(C_\mu)^\perp=\bH^0(C_{\mu_\C})^\perp
  \end{equation*}
  Then, we see
  that
  \begin{equation*}
    d_{(0,0,\theta(x))}T(u'',0)=D''_{\theta(x)}u''=d_{(0,0,\theta(x))}T(0,(u'')^\#_{\theta(x)})
  \end{equation*}
  so that $u''=0$.
\end{proof}

\subsection{Perturbed Kuranishi maps}
The Hitchin-Kobayashi correspondence characterizes polystable orbits
in $\B^{ss}$ via the moment map $\mu$. Since $\theta$ should
eventually induce a local chart for the moduli space, we should be
able to relate the polystable orbits in $\bH^1$ with respect to the
complex reductive group $K^\C$ to the polystable orbits in
$\B$. Therefore, we would like to pullback the moment map $\mu$ to
$U\subset\bH^1$ by $\theta$ and then use the pullback moment map
$\theta^*\mu$ to characterize polystable orbits in $U$. However,
$\theta^*\mu$ takes values in $\Omega^2(\g_E)_{k-1}$ instead of
$\bH^2(C_\mu)\cong\bH^0(C_\mu)^*$. To fix this issue, we will perturb
the Kuranishi map along $\G^\C$-orbits in the following way.
\begin{lemma}\label{sec:kuran-local-models-perturbation}
  If $U\subset\bH^1$ is sufficiently small, then there is a unique
  smooth function $\beta$ defined on $U$ and taking values in an open
  neighborhood of $0$ in $\bH^2(C_\mu)^\perp_{k+1}$ such that the perturbed
  Kuranishi map $\Theta:=\theta e^{-i*\beta}$ is smooth and
  $K$-equivariant, and $\nu:=\Theta^*\mu$ takes values in
  $\bH^2(C_\mu)$ and hence is a moment map for the $K$-action on $U$
  with respect to the symplectic form $\Theta^*\Omega_I$. Moreover,
  the derivative of $\Theta$ at $0$ is the inclusion map.
\end{lemma}
Before giving the proof, we remark that the perturbed Kuranshi map
$\Theta$ is no longer holomorphic and hence the form
$\Theta^*\Omega_I$ is no longer K\"ahler.
\begin{proof}
  We follow the proof of \cite[Proposition
  7]{szekelyhidi2010kahler}. Consider the map
  \begin{gather*}
    L\colon U\times\bH^2(C_\mu)^\perp_{k+1}\to\bH^2(C_\mu)^\perp_{k-1}\\
    L(x,\beta)=(1-H)\mu(\theta(x)e^{-i*\beta})
  \end{gather*}
  where $H$ is the harmonic projection defined in $C_\mu$. Then, the
  derivative of $L$ at $(0,0)$ along the direction $(0,\beta)$ is
  given by
  \begin{equation*}
    d_{(0,0)}L(0,\beta)=(1-H)d_2(-Id_1*\beta)=d_2d_2^*\beta
  \end{equation*}
  where the second equality follows from the formula
  $d_2^*=-Id_1*$. Since
  \begin{equation*}
    d_2d_2^*\colon\bH^2(C_\mu)^\perp_{k+1}\to\bH^2(C_\mu)^\perp_{k-1}
  \end{equation*}
  is an isomorphism, the implicit function theorem guarantees the
  existence of the desired function $\beta$. Since $L$ is
  $K$-equivariant, the uniqueness of $\beta$ implies that $\Theta$ is
  also $K$-equivariant. A direct computation shows that $d_0\Theta$ is
  the inclusion map.
\end{proof}
Although we cannot prove a local slice theorem for the $\G^\C$-action,
the following is a substitute that relates the polystability of Higgs
bundles to that of points in $\bH^1$ with respect to the
$K^\C$-action.
\begin{theorem}\label{sec:kuran-local-models-closedness-polystability}
  If $U$ is sufficiently small, then the induced map
  \begin{equation*}
    U\times_K\G_{k+1}\to\sC_k\qquad [x,g]\mapsto\Theta(x)g
  \end{equation*}
  is injective. Moreover, there is an open ball $B\subset U$ around
  $0$ in the $L^2$-norm such that the following are equivalent for every $x\in B\cap Z$:
  \begin{enumerate}
  \item $xK^\C$ is closed in $\bH^1$.
    
  \item $xK^\C\cap\nu^{-1}(0)\neq\emptyset$.
  \end{enumerate}
\end{theorem}
\begin{proof}
  The derivative of the induced map at $[0,1]$ is given by
  \begin{equation*}
    \bH^1\oplus\bH^0(C_\mu)^\perp_{k+1}\to\Omega^{0,1}(\g_E^\C)_k\oplus\Omega^{1,0}(\g_E^\C)_k\qquad
    (x,u)\mapsto x+D''u
  \end{equation*}
  Since it is injective, we see that the induced map is locally
  injective around $[0,1]$. Then, we assume to the contrary that such
  $U$ does not exist. Therefore, there are sequences $[x_n,g_n]$ and
  $[x_n',g_n']$ such that
  \begin{enumerate}
  \item $x_n,x_n'$ converge to $0$ in $\bH^1$.

  \item $\Theta(x_n)g_n=\Theta(x_n')g_n'$.
    
  \item $[x_n,g_n]\neq[x_n',g_n']$ for all $n$.
  \end{enumerate}
  Since the $\G_{k+1}$-action is proper, by passing to a subsequence,
  we may assume that $g_n'g_n^{-1}$ converges to some
  $g\in\G_{k+1}$. Letting $n\to\infty$, we see that
  $\Theta(0)=\Theta(0)g$ so that $g\in K$. Now, on the one hand,
  $[x_n',g_n'g_n^{-1}]\neq[x_n,1]$ for any $n$. On the other hand,
  both $[x_n',g_n'g_n^{-1}]$ and $[x_n,1]$ converge to $[0,1]$ so that
  they are equal when $n\gg0$, since the induced map is locally
  injective around $[0,1]$. This is a contradiction.

  Now, we prove the second part of the proposition. By
  Proposition~\ref{sec:kuran-local-models-local-completeness}, there
  are open neighborhoods $N_1\times N_2\times V$ of $(0,0,A,\Phi)$ and
  $W$ of $(A,\Phi)$ such that $T\colon N_1\times N_2\times V\to W$ is
  a homeomorphism. Here, $V$ and $W$ are open subsets in
  $\B^{ss}_k$. If $U$ is sufficiently small, $\theta\colon Z\to V$ is
  a homeomorphism so that
  Proposition~\ref{sec:kuran-local-models-infi-slice-theorem}
  holds. Let $O$ be an open neighborhood of $0$ in
  $\bH^2(C_\mu)^\perp_{k+1}$ such that the smooth function
  $\beta\colon U\to O$ and hence $\Theta:=\theta e^{-i*\beta}$ are
  defined. By shrinking $N_2$ if necessary, we may assume that
  $N_2\subset O$. Then, by \cite[Proposition 3.7]{Wilkin2006}, there
  is an open neighborhood $W'\subset W$ of $(A,\Phi)$ in $\B^{ss}_k$
  such that the Yang-Mills-Higgs flow starting at any Higgs bundle
  inside $W'$ stays and converges in $W$. Moreover, we may assume that
  $T(N'_1\times N_2'\times V')=W'$ for some open neighborhood
  $N'_1\times N'_2\times V'\subset N_1\times N_2\times V$ of
  $(0,0,A,\Phi)$ such that $\theta\colon Z\cap B\to V'$ for some open
  ball $B\subset U$ around $0$.

  Now, suppose $x\in B\cap Z$ is such that $xK^\C$ is closed in
  $\bH^1$. Let $(B_t,\Psi_t)$ be the gradient flow starting at
  $\theta(x)$. By the previous setup, $\theta(x)\in V'\subset W'$ so
  that $(B_t,\Psi_t)$ stays in $W$. Therefore, we may write
  $(B_t,\Psi_t)=\theta(x_t)e^{-i*\beta_t}e^{u_t}$ for some $x_t\in Z$
  and $(u_t,\beta_t)\in N_1\times N_2$. We claim that $x_t$ stays in
  the $K^\C$-orbit of $x$. Since the gradient of $\|\mu\|^2$ is
  tangent to $\G^\C_{k+1}$-orbits, we may write
  $d_x\theta(\dot{x}_t)=(u_t)^\#_{\theta(x_t)}$ for some
  $u_t\in\Omega^0(\g_E)_{k+1}$ that depends on $t$ smoothly. Here,
  $u_t^\#$ is the infinitesimal action of $u_t$. Then,
  Proposition~\ref{sec:kuran-local-models-infi-slice-theorem} implies
  that $u_t\in\bH^0(C_{\mu_\C})$ and $\dot{x}_t=(u_t)^\#_{x_t}$. On
  the other hand, the ordinary differential equation in $K^\C$,
  \begin{equation*}
    g_t^{-1}\dot{g}_t=u_t\qquad g_0=1
  \end{equation*}
  has a unique solution $g_t\in K^\C$. By the uniqueness, we see that
  $x_t=xg_t$. Therefore, the claim follows. Then, the fact that $T$ is
  a homeomorphism implies that both $x_t$, $\beta_t$ and $u_t$
  converge. Therefore, letting $t\to\infty$, we have
  $\theta(x_\infty)e^{-i*\beta_\infty}e^{u_\infty}=(B_\infty,\Psi_\infty)$
  and $\mu(B_\infty,\Psi_\infty)=0$. Since $e^{u_\infty}\in\G_{k+1}$,
  $\theta(x_\infty)e^{-i*\beta_\infty}\in\mu^{-1}(0)$. Since
  $N_2\subset O$, the uniqueness of $\beta$ in
  Lemma~\ref{sec:kuran-local-models-perturbation} implies that
  $\beta(x_\infty)=\beta_\infty$. Hence,
  \begin{equation*}
    \Theta(x_\infty)=\theta(x_\infty)e^{-i*\beta_\infty}\in\mu^{-1}(0)
  \end{equation*}
  Finally, since $xK^\C$ is closed in $\bH^1$, we see that
  $x_\infty\in xK^\C$. Again, by the previous setup,
  $x_\infty\in Z\subset U$.

  Conversely, suppose $xK^\C$ is not closed in $\bH^1$. Since the
  $K$-action on $\bH^1$ is linear and preserves the $L^2$-metric, it
  admits a standard moment map $\nu_0$ such that $\nu_0(0)=0$. Since
  $(\grad\|\cdot\|_{L^2}^2,\grad\|\nu_0\|^2)_{L^2}=8\|\nu_0\|^2$ (see
  \cite[Example 2.3]{sjamaar1995holomorphic}), the gradient flow of
  $\|\nu_0\|^2$ starting at $x$ stays in $B$ and converges to some
  $y\in B\cap Z$ such that $\nu_0(y)=0$. By the Kempf-Ness theorem,
  $yK^\C$ is closed in $\bH^1$. Of course,
  $y\in\bar{xK^\C}\setminus xK^\C$. Hence, by the previous paragraph,
  we can find $y_\infty\in yK^\C\cap U$ such that
  $\mu(\Theta(y_\infty))=0$. Hence, we have
  \begin{equation*}
    \Theta(y_\infty)\sim_{\G_{k+1}^\C}\Theta(y)\in\bar{\Theta(x)\G_{k+1}^\C}
  \end{equation*}
  where $\sim_{\G_{k+1}^\C}$ is the equivalence relation generated by
  the $\G_{k+1}^\C$-action. Now, since $xK^\C$ contains a zero of
  $\nu$ in $U$, we may assume that $\mu(\Theta(x))=0$. Then, the
  following
  Lemma~\ref{sec:kuran-local-models-unique-polystable-orbit}
  implies that $\Theta(y_\infty)\sim_{\G_{k+1}^\C}\Theta(x)$ so that
  $\Theta(y_\infty)\sim_{\G_{k+1}}\Theta(x)$ by the Hitchin-Kobayashi
  correspondence. Then, the injectivity of $[x,g]\mapsto\Theta(x)g$
  implies that $y_\infty\sim_K x$. This is a contradiction.
\end{proof}

The following result is nothing but the fact that the closure of the
$\G_{k+1}^\C$-orbit of a semistable Higgs bundle contains a unique
polystable orbit. Since we cannot find a proof in the literature, we
provide one here:
\begin{lemma}\label{sec:kuran-local-models-unique-polystable-orbit}
  Let $(B,\Psi)$ be a semistable Higgs bundle. If
  $(B_i,\Psi_i)\in\bar{(B,\Psi)\G_{k+1}^\C}$ $(i=1,2)$ are
  polystable Higgs bundles, then
  $(B_1,\Psi_1)\sim_{\G_{k+1}^\C}(B_2,\Psi_2)$.
\end{lemma}
\begin{proof}
We may assume that $\mu(B_i,\Psi_i)=0$ for $i=1,2$. Let
$r\colon\B^{ss}_k\to\mu^{-1}(0)$ be the retraction (see \cite[Theorem
1.1]{Wilkin2006}) given by the Yang-Mills-Higgs flow. Suppose
there are sequences $(B_i^j,\Psi_i^j)\in(B,\Psi)\G^\C_{k+1}$ such that
$(B_i^j,\Psi_i^j)\xrightarrow{j\to\infty}(B_i,\Psi_i)$. By the
openness of $\B^{ss}_k$, each $(B_i^j,\Psi_i^j)$ is semistable if
$j\gg0$. By the continuity of $r$, we have
\begin{equation*}
  r(B_i^j,\Psi_i^j)\xrightarrow{j\to\infty}r(B_i,\Psi_i)=(B_i,\Psi_i)
\end{equation*}
By \cite[Theorem 1.4]{Wilkin2006}, we see that each
$r(B_i^j,\Psi_i^j)$ is the graded object of the Seshadri filtration
of $(B_i^j,\Psi_i^j)$. Since graded objects are determined by
$\G_{k+1}^\C$-orbits, we conclude that
\begin{equation*}
  r(B_1^j,\Psi_1^j)\sim_{\G_{k+1}^\C}\mathrm{Gr}(B,\Psi)\sim_{\G_{k+1}^\C}r(B_2^l,\Psi_2^l)
\end{equation*}
for each $j,l$ so that $r(B_1^j,\Psi_1^j)\sim_{\G_{k+1}}
r(B_2^l,\Psi_2^l)$. Since the $\G_{k+1}$-action is proper, $\G_{k+1}$-orbits are
closed. Letting $j\to\infty$, we see that $(B_1,\Psi_1)\in
r(B_2^l,\Psi_2^l)\G_{k+1}$. Now, letting $l\to\infty$, we see that
$(B_1,\Psi_1)\sim_{\G_{k+1}}(B_2,\Psi_2)$.
\end{proof}

\subsection{Open embeddings into the moduli space}

Let $\Z:=Z\cap B$ which is a closed complex subspace of $B$. Note that
$\Z$ is $K$-invariant but not $K^\C$-invariant. To fix this issue,
recall that every open ball around $0$ (in the $L^2$-norm) in the
in $\bH^1$ is $K$-invariant and orbit-convex (see \cite[Definition 1.2 and Lemma
1.14]{sjamaar1995holomorphic}). By \cite[\S3.3,
Proposition]{heinzner1991geometric}, $\Z K^\C$ is a closed complex
subspace of $BK^\C$, and $\Z$ is open in $\Z K^\C$. Recall the
standard moment map $\nu_0\colon\bH^1\to\bH^2(C_\mu)$ used in the
proof of
Theorem~\ref{sec:kuran-local-models-closedness-polystability}. Then,
by the analytic GIT developed in \cite{heinzner1994reduction} or
\cite[\S0]{heinzner1996kahlerian}, there is a categorical quotient
$\pi\colon \Z K^\C\to \Z K^\C\sslash K^\C$ in the category of reduced
complex spaces such that every fiber of $\pi$ contains a unique closed
$K^\C$-orbit, and the inclusion
$\nu_0^{-1}(0)\cap \Z K^\C\hookrightarrow \Z K^\C$ induces a
homeomorphism
\begin{equation*}
  (\nu_0^{-1}(0)\cap \Z K^\C)/K\xrightarrow{\sim} \Z K^\C\sslash K^\C
\end{equation*}
Moreover, as a topological space, $\Z K^\C\sslash K^\C$ is the quotient
space defined by the equivalence relation that $x\sim y$ if and only
if $\bar{xK^\C}\cap\bar{yK^\C}\neq\emptyset$.

A corollary of
Theorem~\ref{sec:kuran-local-models-closedness-polystability} is
that $\Z K^\C\sslash K^\C$ can be realized as a singular symplectic quotient
with respect to $\nu$ instead of $\nu_0$.
\begin{corollary}\label{sec:kuran-local-models-symplectic-quotient-nu-ZKC}
  The inclusion $j\colon\nu^{-1}(0)\cap\Z K^\C\hookrightarrow\Z K^\C$ induces
  a homeomorphism
  \begin{equation*}
    \bar{j}\colon(\nu^{-1}(0)\cap\Z K^\C)/K\xrightarrow{\sim}\Z K^\C\sslash K^\C
  \end{equation*}
  As a consequence, the perturbed Kuranishi map $\Theta$ induces
  well-defined continuous maps $\bar{\Theta}$ and $\varphi$ in the
  following commutative diagram
  \begin{equation*}
    \xymatrix
    {
      \Z K^\C\sslash K^\C\ar[r]^-\varphi &\B^{ps}_k/\G_{k+1}^\C\\
      (\nu^{-1}(0)\cap\Z K^\C)/K\ar[r]^-{\bar{\Theta}}\ar[u]^-\sim &(\mu^{-1}(0)\cap\B_k)/\G_{k+1}\ar[u]^-\sim
    }
  \end{equation*}
  More explicitly, $\varphi$ is given by the formula
  \begin{equation*}
    \varphi[x]=[r\theta(x)]\qquad x\in\Z
  \end{equation*}
  where $r\colon\B_k^{ss}\to\mu^{-1}(0)$ is the retraction defined by
  the Yang-Mills-Higgs flow.
\end{corollary}
\begin{proof}
  Clearly, $\bar{\Theta}$ is a well-defined continuous map. To define
  $\varphi$, it suffices to show that $\bar{j}$ is a
  homeomorphism. Therefore, we show that it has a continuous inverse
  and follow the notations and the setup in the proof of
  Theorem~\ref{sec:kuran-local-models-closedness-polystability}. Let
  $\pi\colon\Z K^\C\to\Z K^\C\sslash K^\C$ be the quotient map. If
  $xg\in\Z K^\C$ with $x\in\Z$, by using the gradient flow of
  $\|\nu_0\|^2$, we see that there is a closed $K^\C$-orbit
  $\tilde{x}K^\C\subset\bar{xK^\C}$ with $\tilde{x}\in\Z$. Then,
  Theorem~\ref{sec:kuran-local-models-closedness-polystability}
  implies that there exists
  \begin{equation*}
    x_\infty\in\nu^{-1}(0)\cap\tilde{x}K^\C\subset\nu^{-1}(0)\cap\bar{xK^\C}
  \end{equation*}
  Therefore, if $\pi(xg)=\pi(yh)$, then $\pi(x_\infty)=\pi(y_\infty)$ so
  that
  \begin{equation*}
    \bar{\Theta(x_\infty)\G_{k+1}^\C}\cap\bar{\Theta(y_\infty)\G_{k+1}^\C}\neq\emptyset
  \end{equation*}
  If we can show that $x_\infty\sim_Ky_\infty$, then the map
  \begin{equation*}
    \bar{j}^{-1}\colon\Z K^\C\sslash K^\C\to(\nu^{-1}(0)\cap\Z K^\C)/K\qquad[xg]\mapsto[x_\infty]
  \end{equation*}
  is well-defined. Now, $x_\infty\sim_K y_\infty$ follows from the
  following Lemma.
  \begin{lemma}
    If $(A_i,\Phi_i)$ $(i=1,2)$ are Higgs bundles such that
    $\mu(A_i,\Phi_i)=0$ and
    $\bar{(A_1,\Phi_1)\G_{k+1}^\C}\cap\bar{(A_2,\Phi_2)\G_{k+1}^\C}\neq\emptyset$,
    then $(A_1,\Phi_1)\sim_{\G_{k+1}}(A_2,\Phi_2)$.
  \end{lemma}
  \begin{proof}
    Let $(B,\Psi)$ be a Higgs bundle in the intersection of the
    closures. Hence, there is a sequence
    $(A_i^j,\Phi_i^j)\in(A_i,\Phi_i)\G_{k+1}^\C$ converging to $(B,\Psi)$. The
    continuity of $r$ implies that
    $r(A_i^j,\Phi_i^j)\xrightarrow{j\to\infty} r(B,\Psi)$. By
    \cite[Theorem 1.4]{Wilkin2006},
    \begin{equation*}
      r(A_i^j,\Phi_i^j)\sim_{\G_{k+1}^\C}Gr(A_i,\Phi_i)=(A_i,\Phi_i)
    \end{equation*}
    so that $r(A_i^j,\Phi_i^j)\sim_{\G_{k+1}}(A_i,\Phi_i)$. Hence, there
    is a sequence of $g_i^j\in\G$ such that
    $(A_i,\Phi_i)g^j_i\xrightarrow{j\to\infty} r(B,\Psi)$. Since the
    $\G_{k+1}$-action is proper, by passing to a subsequence, we may
    assume that $g_i^j\xrightarrow{j\to\infty}g_i$ for some
    $g_i\in\G_{k+1}$. Hence, $(A_i,\Phi_i)g_i=r(B,\Psi)$.
  \end{proof}
  Continuing with the proof of
  Corollary~\ref{sec:kuran-local-models-symplectic-quotient-nu-ZKC},
  we show that $\bar{j}^{-1}$ is continuous. Recall that $x_\infty$ is
  determined by the equation
  $\theta(x_\infty)e^{-i*\beta_\infty}e^{u_\infty}=r(\theta(\tilde{x}))$. By
  the continuity of $r$, $T^{-1}$ and $\theta^{-1}$, we see that the
  map $\Z\ni \tilde{x}\mapsto x_\infty$ is continuous. Moreover,
  $\Z\ni x\mapsto \tilde{x}$ is also continuous, which is a general
  property of the gradient flow of $\|\nu_0\|^2$. Since $\Z$ is open
  in $\Z K^\C$, we conclude that $\bar{j}^{-1}$ is continuous.

  It remains to show that $\bar{j}^{-1}$ is indeed the inverse of
  $\bar{j}$. If $xg\in\nu^{-1}(0)\cap \Z K^\C$ with $x\in\Z$, then
  $xK^\C$ is closed in $\bH^1$
  (Theorem~\ref{sec:kuran-local-models-closedness-polystability}). Since
  $\bar{j}^{-1}$ is well-defined, we see that
  \begin{equation*}
    (xg)_\infty\sim_Kx_\infty\sim_{K^\C}\tilde{x}\sim_{K^\C}x\sim_{K^\C}xg
  \end{equation*}
  Then, $\nu((xg)_\infty)=\nu(xg)=0$ implies that
  $(xg)_\infty\sim_Kxg$. Conversely, if $xg\in\Z K^\C$ with $x\in\Z$,
  then $x_\infty\in\bar{xK^\C}$ so that $\pi(xg)=\pi(x_\infty)$.

  Finally, to obtain a formula for $\varphi$, note that
  \begin{equation*}
    \Theta(x_\infty)\in\bar{\Theta(x)\G_{k+1}^\C}=\bar{\theta(x)\G_{k+1}^\C}
  \end{equation*}
  Moreover, $r(\theta(x))\in\bar{\theta(x)\G_{k+1}^\C}$. Hence, by
  Lemma~\ref{sec:kuran-local-models-unique-polystable-orbit},
  $\Theta(x_\infty)\sim_{\G_{k+1}^\C}r(\theta(x))$.
\end{proof}

The next result shows that $\Z K^\C\sslash K^\C$ is a local model for
the quotient $\M_k=\B^{ps}_k/\G_{k+1}^\C$. Strictly speaking, $\M_k$
is not the moduli space $\M$. That said, there is a natural map
$\M\to\M_k$. Note that \cite[Lemma 14.8]{Atiyah1983} and the elliptic
regularity for $\bar{\partial}_A$ with $A\in\A$ imply that every point
in $\M_k$ has a $C^\infty$ representative. As a consequence, the
natural map $\M\to\M_k$ is surjective. Its injectivity follows from
\cite[Lemma 14.9]{Atiyah1983}. Later, as a consequence of
Theorem~\ref{sec:kuran-local-models-local-charts}, we will show that
$\M\to\M_k$ is a homeomorphism, which justifies our use of Sobolev
completions.
\begin{theorem}\label{sec:kuran-local-models-local-charts}
  If $B$ is sufficiently small,
  $\varphi\colon\Z K^\C\sslash K^\C\to\M_k$ is an open embedding.
\end{theorem}
\begin{proof}
  We will follow the notations and the setup in the proof of
  Theorem~\ref{sec:kuran-local-models-closedness-polystability}. Since
  $\bar{\Theta}$ is injective, $\varphi$ is injective. Let
  $\Pi\colon\B^{ps}_k\to\M_k$ be the quotient map, and consider the open
  set $O=\Pi(W'\cap\B^{ps}_k)$. If $(B,\Psi)\in W'\cap\B^{ps}_k$, then
  $(B,\Psi)=\theta(x)e^{-i*\beta}e^u$ for some $x\in\Z$. We claim that
  $\varphi[x]=[B,\Psi]$. By the construction of $\varphi$ in the proof
  of
  Corollary~\ref{sec:kuran-local-models-symplectic-quotient-nu-ZKC},
  we see that $\varphi[x]=[\Theta(x_\infty)]$ for some
  $x_\infty\in\nu^{-1}(0)\cap\Z K^\C\cap\bar{xK^\C}$ so that
  \begin{equation*}
    \Theta(x_\infty)\in\bar{\theta(x)\G_{k+1}^\C}=\bar{(B,\Psi)\G_{k+1}^\C}
  \end{equation*}
  By Lemma~\ref{sec:kuran-local-models-unique-polystable-orbit},
  we have $\Theta(x_\infty)\sim_{\G_{k+1}^\C}(B,\Psi)$. As a consequence,
  the open set $O$ is contained in the image of $\varphi$. Hence, we
  obtain a bijective continuous map $\varphi\colon\tilde{O}\to O$,
  where $\tilde{O}=\varphi^{-1}(O)$.

  To show that $\varphi|_{\tilde{O}}$ is a homeomorphism, we will show
  that its inverse is continuous. From the previous paragraph, we see
  that its inverse should be $[B,\Psi]\mapsto[x]$. The continuity
  follows from the continuity of $\theta^{-1}$ and
  $T^{-1}$. Therefore, it remains to prove that it is well-defined. If
  $(B',\Psi')\in W'\cap\B^{ps}_k$ lies in the $\G_{k+1}^\C$-orbit of
  $(B,\Psi)$, then
  \begin{equation*}
    \Theta(x_\infty)\sim_{\G_{k+1}^\C}(B,\Psi)\sim_{\G_{k+1}^\C}(B',\Psi')\sim_{\G_{k+1}^\C}\Theta(x_\infty')
  \end{equation*}
  so that
  \begin{equation*}
    \bar{xK^\C}\ni x_\infty\sim_K x_\infty'\in\bar{x'K^\C}
  \end{equation*}
  Hence, $\bar{xK^\C}\cap\bar{x'K^\C}\neq\emptyset$.

  Finally, we show that if $B$ is sufficiently small, then $\varphi$
  is an open embedding. Write $\pi^{-1}(\tilde{O})=\Z K^\C\cap Q$ for
  some open set $Q$ in $\bH^1$, where
  $\pi\colon \Z K^\C\to \Z K^\C\sslash K^\C$ is the quotient
  map. Since $0\in Q$, choose some open ball $B'\subset Q\cap B$
  around $0$. By \cite[Lemma 1.14]{sjamaar1995holomorphic}, we know
  that $B$ and $B'$ are $\nu_0$-convex (see \cite[(2.6),
  Definition]{heinzner1994reduction}). Hence, by definition of $\Z$,
  $\Z$ is also $\nu_0$-convex. Hence, by \cite[(3.1),
  Lemma]{heinzner1994reduction}, we see that
  $\Z K^\C\cap B'K^\C=(\Z\cap B')K^\C$. Then, we claim that
  $(\Z\cap B')K^\C\subset\pi^{-1}(\tilde{O})$. In fact, if
  $xg\in(\Z\cap B')K^\C$ with $x\in\Z\cap B$, then $x\in ZK^\C\cap
  Q$. Since $ZK^\C\cap Q$ is $K^\C$-invariant, $xg\in ZK^\C\cap
  Q$. Finally, we claim that $(\Z\cap B')K^\C$ is also $\pi$-saturated
  so that $(\Z\cap B')K^\C\sslash K^\C$ is an open neighborhood of $[0]$ in
  $\Z K^\C\sslash K^\C$. Therefore, if $B$ is shrunk to $B'$, and $\Z$
  is shrunk to $\Z\cap B'$, we see that $\varphi$ is an open
  embedding.

  Suppose $\pi(xg)=\pi(yh)$ for some $x\in\Z$ and $y\in\Z\cap B'$. We
  want to show that $xg\in(\Z\cap B')K^\C$. By using the gradient flow
  of $\|\nu_0\|^2$, we can find a closed orbit
  $y'K^\C\subset\bar{yK^\C}$ with $y'\in\Z\cap B'$. Since every fiber
  of $\pi$ contains a unique closed orbit,
  $y'K^\C\subset\bar{xK^\C}$. Since $B'$ is open,
  $xK^\C\cap B'\neq\emptyset$. Hence,
  $x\in B'K^\C\cap\Z K^\C=(\Z\cap B')K^\C$.
\end{proof}
To show that $\M\to\M_k$ is a homeomorphism, we need the
following lemma.
\begin{lemma}\label{sec:kuran-local-models-regularity-1}
  Elements in $\B_k\cap[(A,\Phi)+\ker (D'')^*]$ are of class $C^\infty$.
\end{lemma}
\begin{proof}
  Suppose $(D'')^*(\alpha'',\eta)=0$ and
  $(\bar{\partial}_A+\alpha'')(\Phi+\eta)=0$, where $\alpha''$ is the
  $(0,1)$-component of $\alpha$. The second equation is also
  equivalent to $D''(\alpha'',\eta)+[\alpha'',\eta]=0$. Hence,
  $\Delta(\alpha,\eta)=-(D'')^*[\alpha'',\eta]$ where
  $\Delta=D''(D'')^*+(D'')^*D''$ is the Laplacian defined in
  $C_{\mu_\C}$. Since $k>1$, the Sobolev multiplication theorem (see
  \cite[Theorem 4.4.1]{friedman1998gauge}) implies that
  $[\alpha'',\eta]$ is in $L_k^2$ and hence $(D'')^*[\alpha'',\eta]$
  is in $L_{k-1}^2$. By the elliptic regularity, $(\alpha'',\eta)$ is
  hence in $L_{k+1}^2$. By induction, $(\alpha'',\eta)$ is in
  $C^\infty$.
\end{proof}
\begin{lemma}\label{sec:kuran-local-models-regularity-2}
  The map $\varphi$ in
  Corollary~\ref{sec:kuran-local-models-symplectic-quotient-nu-ZKC}
  factors through the natural map $\M\to\M_k$.
\end{lemma}
\begin{proof}
  Recall that the formula for $\varphi$ is given by
  $\varphi[x]=[r\theta(x)]$ where $x\in\Z$. By
  Lemma~\ref{sec:kuran-local-models-regularity-1}, $\theta$ restricts
  to a continuous map $Z\to\B^{ss}\cap((A,\Phi)+\ker (D'')^*)$. Since
  $r\colon\B^{ss}\to\mu^{-1}(0)$ is continuous,
  $\Z\ni x\mapsto[r\theta(x)]\in\M$ is continuous. Finally,
  \cite[Lemma 14.9]{Atiyah1983} and the fact that $\varphi$ is
  well-defined imply that $\varphi$ factors through $\M\to\M_k$.
\end{proof}
\begin{corollary}\label{sec:kuran-local-models-Cinfty-Sobolev-homeo}
  The natural map $\M\to\M_k$ is a homeomorphism. Therefore, the map
  $\varphi\colon\Z K^\C\sslash K^\C\to\M$ is an open embedding.
\end{corollary}
\begin{proof}
  By Lemma~\ref{sec:kuran-local-models-regularity-2} and
  Theorem~\ref{sec:kuran-local-models-local-charts}, $\M\to\M_k$ is
  locally an open map and hence open.
\end{proof}

\section{Gluing local models}\label{sec:gluing-local-models}
For the rest of the paper, we will drop the subscripts that indicate
Sobolev completions for notational convenience. By
Lemma~\ref{sec:kuran-local-models-regularity-1},
\ref{sec:kuran-local-models-regularity-2} and
Corollary~\ref{sec:kuran-local-models-Cinfty-Sobolev-homeo}, this
should not cause any confusion. The main result in this section is the
following, which is part of
Theorem~\ref{sec:introduction-is-a-complex-space}. The normality of
$\M$ will be proved in Lemma~\ref{sec:holomorphicity-normality}.
\begin{theorem}\label{sec:gluing-local-models-complex-space-structure}
  The moduli space $\M$ is a complex space locally biholomorphic to a
  Kuranishi local model $\Z K^\C\sslash K^\C$.
\end{theorem}

Let $(A_i,\Phi_i)$ ($i=1,2$) be Higgs bundles such that
$\mu(A_i,\Phi_i)=0$. We will use subscript $i$ to denote relevant
objects associated with $(A_i,\Phi_i)$. Let $\Z_i$ be their Kuranishi
spaces and $\Z_i K_i^\C\sslash K_i^\C$ Kuranishi local models, where
$K_i$ is the $\G$-stabilizer of $(A_i,\Phi_i)$. Let
\begin{equation*}
  \varphi_i\colon \Z_i K_i^\C\sslash K_i^\C\xrightarrow{\sim}O_i\subset\M
\end{equation*}
be the map constructed in
Theorem~\ref{sec:kuran-local-models-local-charts} such that
$O_1\cap O_2\neq\emptyset$. Hence, the transition function is given by
\begin{equation*}
  \varphi_2^{-1}\varphi_1\colon\varphi_1^{-1}(O_1\cap
  O_2)\to\varphi_2^{-1}(O_1\cap O_2)
\end{equation*}
Our goal is to show that $\varphi_2^{-1}\varphi_1$ is holomorphic so
that $\M$ is a complex space. Since holomorphicity is a local
condition, the idea is that the transition function
$\varphi_2^{-1}\varphi_1$ should be locally induced by a holomorphic
$K_1^\C$-invariant map from an open set in $\Z_1 K_1^\C$ to
$\Z_2 K_2^\C\sslash K_2^\C$. Then, the rest of the argument follows
from the universal property of the quotient map
$\pi_i\colon \Z_i K_i^\C\to \Z_i K_i^\C\sslash K_i^\C$. Here, the
technical difficulty is to find an appropriate open set in
$\Z_1 K_1^\C$ that is also $\pi_1$-saturated. This will be overcome in
the following Lemma~\ref{sec:gluing-local-models-saturation}.

To proceed, we follow the notations and the setup in the proof of
Theorem~\ref{sec:kuran-local-models-closedness-polystability}. Let
$[x]\in\varphi_1^{-1}(O_1\cap O_2)$. Using the gradient flow of
$\|\nu_0\|^2$, we may assume that $x\in\Z_1$ has a closed
$K_1^\C$-orbit. Hence, $\theta(x)$ is polystable
(Theorem~\ref{sec:kuran-local-models-closedness-polystability}),
and $\varphi_1[x]=[r\theta_1(x)]=[\theta_1(x)]$. Similarly, there is
some $x'\in\Z_2$ with closed $K_2^\C$-orbit such that
$\varphi_2[x']=\varphi_1[x]$ so that
$\theta_1(x)\sim_{\G^\C}\theta_2(x')$. Since
$\theta_i\colon\Z_i\to V_i'\subset W_i'$ is a homeomorphism,
$\theta_1(x)\in W_1'\cap W_2'h^{-1}$ for some $h\in\G^\C$.

\begin{lemma}\label{sec:gluing-local-models-saturation}
  There is an open neighborhood $C$ of $x$ in $\Z_1 K_1^\C$ such that
  \begin{enumerate}
  \item $CK_1^\C$ is $\pi_1$-saturated.
    
  \item $\theta_1(C)\subset W_1'\cap W_2'h^{-1}$.
    
  \item $[x]\in\pi_1(C)\subset\varphi_1^{-1}(O_1\cap O_2)$
  \end{enumerate}
\end{lemma}
\begin{proof}
  Since $T_1\colon N_1'\times V_1'\to W_1'$ and $\theta_1\colon\Z_1\to
  V_1'$ are homeomorphisms, there is an open ball $Q$ around $x$ such
  that
  \begin{equation*}
    \theta_1(\Z_1\cap Q)\subset W_1'\cap W_2'h^{-1}
  \end{equation*}
  Since $\Z_1$ is open in $\Z_1K_1^\C$, $(\Z_1\cap Q)K_1^\C$ is open
  in $\Z_1K_1^\C$. Then, set
  \begin{equation*}
    C=\pi_1^{-1}\pi_1(\nu_1^{-1}(0)\cap(\Z_1\cap Q)K_1^\C)\cap(\Z_1\cap Q)K_1^\C
  \end{equation*}
  By
  Corollary~\ref{sec:kuran-local-models-symplectic-quotient-nu-ZKC},
  $C$ is open in $\Z_1K_1^\C$. Clearly, $(2)$ follows and $x\in C$.

  To show that $CK_1^\C$ is $\pi_1$-saturated, let $x\in \Z_1K_1^\C$
  be such that $\pi_1(x)=\pi_1(y)$ for some $y\in C$. By definition of
  $C$, $\pi_1(y)=\pi_1(y')$ for some
  $y'\in \nu_1^{-1}(0)\cap (\Z_1\cap Q)K_1^\C$. Since $y'K_1^\C$ is
  closed, $y'K_1^\C\subset \bar{xK_1^\C}$. Since
  $y'K_1^\C\cap C\neq\emptyset$, and $C$ is open, we conclude that
  $xK_1^\C\cap C\neq\emptyset$. This shows that $(1)$. If $y\in C$,
  then $\pi_1(y)=\pi_1(y'g)$ for some
  $y'g\in\nu_1^{-1}(0)\cap (\Z_1\cap Q)K_1^\C$ with $y'\in\Z_1\cap
  Q$. Therefore, $\varphi_1[y]=[\theta_1(y')]$. By the construction of
  $\varphi_i$ in
  Corollary~\ref{sec:kuran-local-models-symplectic-quotient-nu-ZKC}
  and Theorem~\ref{sec:kuran-local-models-local-charts}, we see that
  \begin{equation*}
    O_i=\Pi r\theta_i(\Z_i)=\Pi r(V_i')=\Pi r(W_i')
  \end{equation*}
  where $\Pi\colon\B^{ps}\to\M$ is the quotient map. Since
  $\theta_1(y')\in W_1'\cap W_2'h^{-1}$ is polystable, it is easy to
  see that $[\theta_1(y')]\in O_1\cap O_2$. This proves $(3)$.
\end{proof}
Now, for $y\in C$, $\theta_1(y)h\in W_2'$. Since $T_2$ is a
homeomorphism, there is $g(y)\in\G^\C$, as a function of $y\in C$,
such that $\theta_1(y)hg(y)\in V_2'$. Hence, we have obtained a map
\begin{equation*}
  \psi_{21}\colon C\to \Z_2K_2^\C\sslash K_2^\C\qquad \psi_{21}(y)=\pi_2\theta_2^{-1}(\theta_1(y)hg(y))
\end{equation*}
\begin{lemma}\ \label{sec:gluing-local-models-holomorphicity}
  \begin{enumerate}
  \item $\psi_{21}$ is holomorphic.
    
  \item If $y,y'\in C$ are in the same $K_1^\C$-orbit, then
    $\psi(y)=\psi(y')$.
  \end{enumerate}
\end{lemma}
\begin{proof}
  Explicitly, we have
  \begin{equation*}
    g(y)=\exp(-p_1T_2^{-1}(\theta_1(y)h))
  \end{equation*}
  where $p_1$ is the projection onto the first factor. Since
  \begin{equation*}
    T_2\colon\bH^0(C_{\mu_\C}^2)^\perp\times((A_2,\Phi_2)+\ker D_2''^*)\to\sC
  \end{equation*}
  is holomorphic, its inverse, when restricted to appropriate open
  neighborhoods, is also holomorphic. Moreover, since the Kuranishi
  map is holomorphic, $\theta_1$ is also holomorphic when the codomain
  is appropriately extended. Therefore, we conclude that $g\colon
  C\to\G^\C$ is holomorphic. Finally, since the $\G^\C$-action is
  holomorphic, we conclude that $\psi_{21}$ is holomorphic.

  To show $(2)$, suppose there are $z,z'\in \Z_2$ such that
  \begin{align*}
    \theta_2(z)&=\theta_1(y)hg(y)\\
    \theta_2(z')&=\theta_1(y')hg(y')
  \end{align*}
  We want to show that $\pi_2(z)=\pi_2(z')$. Since $y$ and $y'$ are in
  the same $K_1^\C$-orbit,
  \begin{equation*}
    \theta_2(z)\sim_{\G^\C}\theta_1(y)\sim_{\G^\C}\theta_1(y')\sim_{\G^\C}\theta_2(z')
  \end{equation*}
  so that $r\theta_2(z)\sim_\G r\theta_2(z')$. This means that
  $\varphi_2[z]=\varphi_2[z']$. Since $\varphi_2$ is injective, $[z]=[z']$.
\end{proof}
\begin{lemma}\label{sec:gluing-local-models-transition-functions}
  The transition function $\varphi_2^{-1}\varphi_1$ is holomorphic.
\end{lemma}
\begin{proof}
  By Lemma~\ref{sec:gluing-local-models-holomorphicity},
  $\psi_{21}$ extends to a $K_1^\C$-invariant holomorphic map
  \begin{equation*}
    \psi_{21}\colon CK_1^\C\to \Z_2K_2^\C\sslash K_2^\C
  \end{equation*}
  Since $CK_1^\C$ is an $\pi_1$-saturated open set
  (Lemma~\ref{sec:gluing-local-models-saturation}),
  \begin{equation*}
    \pi_2\colon CK_1^\C\to\pi_1(CK_1^\C)=:CK_1^\C\sslash K_1^\C
  \end{equation*}
  is also a categorical quotient. As
  a consequence, $\psi_{21}$ decends to a holomorphic map
  \begin{equation*}
    \bar{\psi}_{21}\colon CK_1^\C\sslash K_1^\C\to \Z_2K_2^\C\sslash K_2^\C
  \end{equation*}
  Let $[c]\in CK_1^\C\sslash K_1^\C$ with $c\in C$ and
  $z=\theta_2^{-1}(\theta_1(c)hg(c))$. Hence,
  $\theta_2(z)\sim_{\G^\C}\theta_1(c)$. Therefore,
  \begin{equation*}
    \varphi_2\bar{\psi}_{21}[c]
    =\varphi_2\psi_{21}(c)
    =\varphi_2\pi_2(z)
    =\Pi(r\theta_2(z))
    =\Pi(r\theta_1(z))
    =\varphi_1[c]
  \end{equation*}
  This shows that the transition function $\varphi_2^{-1}\varphi_1$
  coincides with a holomorphic map $\bar{\psi}_{21}$ on an open
  neighborhood $CK_1^\C\sslash K_1^\C$ of $[x]$ in
  $\varphi_1^{-1}(O_1\cap O_2)$. This completes the proof.
\end{proof}
\begin{proof}[Proof of
  Theorem~\ref{sec:gluing-local-models-complex-space-structure}]
  By the properness of $\G$-action, $(\mu^{-1}(0)\cap\B)/\G$ is
  Hausdorff. The Hitchin-Kobayashi correspondence implies that
  $\M$ is Hausdorff. The Kuranishi local models are constructed in
  Corollary~\ref{sec:kuran-local-models-symplectic-quotient-nu-ZKC}
  and Theorem~\ref{sec:kuran-local-models-local-charts}. By
  Lemma~\ref{sec:gluing-local-models-transition-functions}, the
  transition functions are holomorphic.
\end{proof}

\section{Singularities in Kuranishi spaces}\label{sec:sing-kuran-spac}
In this section, we will show that Kuranishi spaces have only cone
singularities. We will use the same notations in
Section~\ref{sec:kuran-local-models}. The main result in this section
is the following (cf. \cite[Theorem
2.24]{huebschmann1995singularities} and \cite[Theorem
3]{arms1981symmetry}).
\begin{theorem}\label{sec:sing-kuran-spac-diagram}
  The following diagram commutes
  \begin{equation*}
    \xymatrix
    {
      \tilde{\B}\cap((A,\Phi)+\ker (D'')^*)\ar[r]^-{F}\ar[d]^-{\mu_\C}
      &\bH^1\ar[dl]^-{\frac{1}{2}H[\cdot,\cdot]}\\
      \bH^2(C_{\mu_\C})
    }
  \end{equation*}
\end{theorem}
\begin{proof}
  By construction of $\tilde{\B}$, the restriction of $\mu_\C$ to
  $\tilde{\B}$ is given by
  \begin{equation*}
    \mu_\C(A+\alpha,\Phi+\eta)=H\mu_\C(A+\alpha,\Phi+\eta)=\frac{1}{2}H[\alpha'',\eta;\alpha'',\eta]=H[\alpha'',\eta]
  \end{equation*}
  where $(A+\alpha'',\Phi+\eta)\in\tilde{\B}$. By definition of the
  Kuranishi space $Z$, it suffices to prove:
  \begin{enumerate}
  \item $H[(\alpha'',\eta),(D'')^*G[\alpha'',\eta;\alpha'',\eta]]=0$.
    
  \item $H[(D'')^*G[\alpha'',\eta;\alpha'',\eta],(D'')^*G[\alpha'',\eta;\alpha'',\eta]]=0$.
  \end{enumerate}
  for any $(\alpha'',\eta)\in\ker (D'')^*$. By K\"ahler's
  identities,
  \begin{equation*}
    H[(\alpha'',\eta),(D'')^*G[\alpha'',\eta;\alpha'',\eta]]
    =\pm iH[(\alpha'',\eta),D'*G[\alpha'',\eta;\alpha'',\eta]]
  \end{equation*}
  and $(\alpha'',\eta)\in\ker D'$. Since $D'$ is a derivation with
  respect to $[\cdot,\cdot]$, we see that
  \begin{equation*}
    H[(\alpha'',\eta),D'*G[\alpha'',\eta;\alpha'',\eta]]=\pm HD'[(\alpha'',\eta),*G[\alpha'',\eta;\alpha'',\eta]]=0
  \end{equation*}
  This proves $(1)$. The same argument shows $(2)$. This completes the
  proof.
\end{proof}
As a corollary, we obtain a description of singularities in the
Kuranishi spaces.
\begin{corollary}\label{sec:sing-kuran-spac-singularities}
  The Kuranishi space $Z$ is an open neighborhood of $0$ in the
  quadratic cone
  \begin{equation*}
    Q=\{x\in\bH^1\colon \frac{1}{2}H[x,x]=0\}
  \end{equation*}
\end{corollary}
\begin{proof}
  This is clear by definition of Kuranishi spaces and
  Theorem~\ref{sec:sing-kuran-spac-diagram}.
\end{proof}
It is easy to see that the complex structures on $\sC$ restrict to
$\bH^1$ so that $\bH^1$ has a linear hyperK\"ahler structure. In
particular, the complex symplectic form $\Omega_\C$ on $\sC$ restricts
to $\bH^1$. Hence, there is a standard complex moment map
$\nu_{0,\C}\colon\bH^1\to\bH^2(C_{\mu_\C})$ for the $K^\C$-action with
respect to the linear complex symplectic structure. More precisely,
$\nu_0$ is defined by
\begin{equation*}
  \langle\nu_{0,\C}(x),\xi\rangle=\frac{1}{2}\Omega_\C(x\cdot\xi,x)\qquad\xi\in\bH^0(C_{\mu_\C})
\end{equation*}
Since $i\colon\bH^1\hookrightarrow\sC$ is $K^\C$-equivariant, and
$\mu_\C$ is a complex moment map, $Hi^*\mu_\C$ is a complex moment map
for the $K^\C$-action on $\bH^1$, where $H$ is the harmonic projection
onto $\bH^2(C_{\mu_\C})$. Since $Hi^*\mu_\C(0)=0$, we see that
$Hi^*\mu_\C=\nu_{0,\C}$. On the other hand,
$Hi^*\mu_\C=\frac{1}{2}H[\cdot,\cdot]$. Hence, $Q$ is the zero set of
the standard complex moment map $\nu_{0,\C}$.

Obviously, $\nu_{0,\C}^{-1}(0)$ is a closed complex subspace of
$\bH^1$. In fact, it is an affine variety. Therefore, the affine GIT
quotient $\nu_{0,\C}^{-1}(0)\sslash K^\C$ exists such that the
inclusion
$\nu_0^{-1}(0)\cap\nu_{0,\C}^{-1}(0)\hookrightarrow\nu_{0,\C}^{-1}(0)$
induces a homeomorphism (see \cite[(1.4)]{heinzner1994reduction})
\begin{equation*}
  (\nu_0^{-1}(0)\cap\nu_{0,\C}^{-1}(0))/K\xrightarrow{\sim}\nu_{0,\C}^{-1}(0)\sslash K^\C
\end{equation*}
Note that $(\nu_0^{-1}(0)\cap\nu_{0,\C}^{-1}(0))/K$ is precisely the
hyperK\"ahler quotient with respect to the standard hyperK\"ahler moment
maps on $\bH^1$. 
\begin{theorem}[=Theorem~\ref{sec:introduction-local-model}]\label{sec:sing-kuran-spac-local-model}
  Let $[A,\Phi]\in\M$ be a point such that $\mu(A,\Phi)=0$ and $\bH^1$
  its deformation space, a harmonic space defined in
  $C_{\mu_\C}$. Then, the following hold:
  \begin{enumerate}
  \item $\bH^1$ is a complex-symplectic vector space.
    
  \item The $\G^\C$-stabilizer $K^\C$ at $(A,\Phi)$ is a complex
    reductive group, acts on $\bH^1$ linearly and preserves the
    complex-symplectic structure on $\bH^1$. Moreover, the
    $K^\C$-action on $\bH^1$ admits a canonical complex moment map
    $\nu_{0,\C}$ such that $\nu_{0,\C}(0)=0$.
    
  \item Around $[A,\Phi]$, the moduli space $\M$ is locally
    biholomorphic to an open neighborhood of $[0]$ in the complex
    symplectic quotient $\nu_{0,\C}^{-1}(0)\sslash K^\C$ which is an
    affine GIT quotient.
  \end{enumerate}
\end{theorem}
\begin{proof}
  It remains to show $(3)$. Since $\Z$ is open in $Z$ which is also
  open in $Q$, we have $\Z K^\C$ is open in $Q$. Since $\Z K^\C$ is
  saturated with respect to the quotient $Q\to Q\sslash K^\C$,
  $\Z K^\C\sslash K^\C$ is an open neighborhood of $[0]$ in
  $Q\sslash K^\C$. The rest follows from
  Theorem~\ref{sec:kuran-local-models-local-charts}
  and~\ref{sec:gluing-local-models-complex-space-structure}.
\end{proof}

\section{The Isomorphism between the analytic and the algebraic
  constructions}\label{sec:comp-betw-analyt}
Let $\M_{an}$ be the moduli space $\B^{ps}/\G^\C$ and $\M_{alg}$ the
coarse moduli space of the semistable Higgs bundles of rank $r$ and
degree $0$, where $r$ is the rank of $E$. By \cite[Theorem 4.7,
Theorem 11.1]{Simpson1994}, $\M_{alg}$ is a normal irreducible
quasi-projective variety. By abusing the notation, we also use
$\M_{alg}$ to mean its analytification. Then, there is a natural
comparison map
\begin{equation*}
  i\colon \M_{an}\to\M_{alg}\qquad[A,\Phi]\mapsto[\E_A,\Phi]_s
\end{equation*}
Here, $(\E_A,\Phi)$ is the Higgs bundle determined by $(A,\Phi)$, and
$[\E_A,\Phi]_s$ means the s-equivalence class of $(\E_A,\Phi)$. We
will prove Theorem~\ref{sec:introduction-comparison} in this
section. By \cite[Proposition 5.1]{Wilkin2006}, we see that $i$ is a
bijection of sets. 

\subsection{Continuity}
The first step toward our goal is to show that $i$ is a
homeomorphism. To this end, we need some preparations. First, we may
assume that the degree of $E$ is sufficiently large. This can be
arranged as follows. Fix a holomorphic line bundle
$\sL=(L,\bar{\partial}_L)$ of degree $d>0$. Here, $L$ is the
underlying smooth line bundle of $\sL$, and $\bar{\partial}_L$ is the
$\bar{\partial}$-operator defined by the holomorphic structure on
$\sL$. We may also fix a Hermitian metric on $L$ so that the Chern
connection of $\bar{\partial}_L$ is $d_L$. Then, there is a map
\begin{equation*}
  \B(E)\to\B(E\otimes L)\qquad (A,\Phi)\mapsto(A\otimes1+1\otimes d_L,\Phi\otimes1)
\end{equation*}
Here, $\B(E)$ and $\B(E\otimes L)$ are the configuration spaces of
Higgs bundles with underlying smooth bundles $E$ and $E\otimes L$,
respectively. Since $(\E,\Phi)$ is (semi)stable if and only if
$(\E\otimes\sL,\Phi)$ is (semi)stable, this map restricts to a map
\begin{equation*}
  \B(E)^{ps}\to\B(E\otimes L)^{ps}
\end{equation*}
and eventually descends to a homeomorphism (in the
$C^\infty$-topology)
\begin{equation*}
  \M_{an}\xrightarrow{\otimes\sL}\M_{an}(rd)
\end{equation*}
where $\M_{an}(rd)=\B^{ps}(E\otimes L)^{ps}/\Aut(E\otimes L)$, and
$rd$ is the degree of $E\otimes L$. On the other hand, there is a
homeomorphism (in the analytic topology) $\M_{alg}\to\M_{alg}(rd)$
given by tensoring by $\sL$. Here, $\M_{alg}(rd)$ is the moduli space
of the semistable Higgs bundles of rank $r$ and degree $rd$ in the
category of schemes. Finally, these maps fit into the following
commutative diagram
\begin{equation*}
  \xymatrix
  {
    \M_{an}\ar[r]^-i\ar[d]^-{\otimes\sL} &\M_{alg}\ar[d]^-{\otimes\sL}\\
    \M_{an}(rd)\ar[r]^-i &\M_{alg}(rd)
  }
\end{equation*}
Therefore, the bottom map is a homeomorphism if and only if the top
one is a homeomorphism.

Now, let us recall Nitsure's construction of $\M_{alg}$ in
\cite{nitsure1991moduli}. By the previous paragraph, we may assume
that the degree $d$ of $E$ is sufficiently large so that if
$(\E_A,\Phi)$ is a semistable Higgs bundle defined by $(A,\Phi)\in\B$
then $\E_A$ is generated by global sections and $H^1(X,\E_A)=0$. Let
$p=d+r(1-g)$ and $Q$ be the Quot scheme parameterizing isomorphism
classes of quotients $\sO_X^p\to\E\to0$, where $\E$ is a coherent
sheaf on $X$ with rank $r$ and degree $d$, and $\sO_X$ is the
structure sheaf of $X$. Let $\sO_{X\times Q}^p\to\U\to0$ be the
universal quotient sheaf on $X\times Q$, and $R\subset Q$ be the
subset of all $q\in Q$ such that
\begin{enumerate}
\item the sheaf $\U_q$ is locally free, and
  
\item the map $H^0(X,\sO_X^p)\to H^0(X,\U_q)$ is an isomorphism.
\end{enumerate}
It is shown that $R$ is open in $Q$. Moreover, Nitsure constructed a
linear scheme $F$ over $R$ such that closed points in $F_q$ correspond
to Higgs fields on $\U_q$ for any $q\in Q$. Let $F^{ss}$ denote the
subset of $F$ consisting of semistable Higgs bundles
$(\sO_X^p\to\E\to0,\Phi)$. It is open in $F$. Moreover, the group
$PGL(p)$ acts on $Q$, and the action lifts to $F$. Finally, Nitsure
showed that the good quotient of $F^{ss}$ by the group $PGL(p)$
exists and is the moduli space $\M_{alg}$.

Following \cite{sibley2019continuity}, if $U$ is an open subset of
$\B^{ss}$ (in the $C^\infty$-topology), a map
$\sigma\colon U\to F^{ss}$ is called a \textit{classifying map} if
$\sigma(A,\Phi)$ is a Higgs bundle isomorphic to $(\E_A,\Phi)$.
\begin{lemma}\label{sec:cont-comp-map-clas-map-exists}
  Fix $(A_0,\Phi_0)\in\B^{ss}$. There exists an open neighborhood $U$
  of $(A_0,\Phi_0)$ in $\B^{ss}$ in the $C^\infty$-topology such that
  a classifying map $\sigma\colon U\to F^{ss}$ exists and is
  continuous with respect to the analytic topology on $F^{ss}$.
\end{lemma}
Before giving the proof, we first show that how it implies the
continuity of $i$.
\begin{corollary}\label{sec:continuity-homeo}
  The comparison map $i\colon\M_{an}\to\M_{alg}$ is a homeomorphism.
\end{corollary}
\begin{proof}
  Fix $[A_0,\Phi_0]\in\M_{an}$ such that $(A_0,\Phi_0)\in\B^{ps}$. By
  Lemma~\ref{sec:cont-comp-map-clas-map-exists}, there exists an open
  neighborhood $U$ of $(A_0,\Phi_0)$ such that a continuous
  classifying map $\sigma\colon U\to F^{ss}$ exists. Composed with the
  categorical quotient $F^{ss}\to\M_{alg}$, which is continuous in the
  analytic topology, we obtain a continuous map $U\to\M_{alg}$. By
  construction, it descends to the restriction of $i$ to the open set
  $\pi(U)$, where $\pi\colon\B^{ps}\to\M_{an}$ is the quotient map.

  To see that $i$ is a homeomorphism, we show that it is proper. Since
  $\M_{alg}$ is locally compact in the analytic topology, $i$ is a
  closed map and hence a homeomorphism. Let $h_{an}$ and $h_{alg}$ be
  the Hitchin fibrations defined on $\M_{an}$ and $\M_{alg}$,
  respectively. Then, we have $h_{alg}i=h_{an}$. It is known that
  $h_{an}$ is a proper map (see \cite[Theorem 8.1]{Hitchin1987b} or
  \cite[Theorem 2.15]{wentworth2016higgs}). As a consequence, if $K$
  is a compact subset in $\M_{alg}$ in the analytic topology, then
  $i^{-1}(K)\subset h_{an}^{-1}h_{alg}(K)$. Since $h_{alg}$ is
  continuous, $h_{alg}(K)$ is compact and hence
  $h_{an}^{-1}h_{alg}(K)$ is compact by the properness of
  $h_{an}$. Since $\M_{alg}$ is a separated scheme, $\M_{alg}$ is
  Hausdorff in the analytic topology. Hence, $K$ is closed and
  $i^{-1}(K)$ is also closed and contained in a compact
  set. Therefore, $i^{-1}(K)$ is compact.
\end{proof}
\begin{proof}[Proof of Lemma~\ref{sec:cont-comp-map-clas-map-exists}]
  The proof is essentially taken from that of \cite[Theorem
  6.1]{sibley2019continuity}. We first show that a classifying map
  $\sigma$ exists and then prove its continuity. Let
  $V_0=\ker\bar{\partial}_{A_0}\subset\Omega^0(E)$. By definition of
  $\bar{\partial}$-operators, $V_0=H^0(X,\E_A)$. Since $H^1(\E_A)=0$,
  the Riemann-Roch theorem implies that $\dim V_0=p$. Hence, by
  choosing a basis for $V_0$, we may identify $V_0$ with
  $\C^p$. Moreover, since $\E_{A_0}$ is generated by global sections,
  the evaluation map
  \begin{equation*}
    X\times V_0\to\E_A\qquad (x,s)\mapsto s(x)
  \end{equation*}
  realizes $\E_{A_0}$ as a quotient of
  $V_0\otimes\sO_X\cong\sO_X^p$. Let $(A,\Phi)$ be another point in
  $\B^{ss}$, and consider the map defined by the composition
  \begin{equation*}
    \pi_A\colon V_A=\bar{\partial}_A\hookrightarrow\Omega^0(E)\to V_0
  \end{equation*}
  where $\Omega^0(E)\to V_0$ is given by the harmonic projection
  defined in the following elliptic complex
  \begin{equation*}
    C(A_0)\colon\Omega^0(E)\xrightarrow{\bar{\partial}_{A_0}}\Omega^{0,1}(E)
  \end{equation*}
  We claim that there exists an open neighborhood $U$ of
  $(A_0,\Phi_0)$ such that $\pi_A$ is an isomorphism for every
  $(A,\Phi)\in U$. Write $\pi_A(s)=s+u_s$ for some $u_s\in V_0^\perp$
  and $\bar{\partial}_A=\bar{\partial}_{A_0}+a$ for some
  $a\in\Omega^{0,1}(\g_E^\C)$. Let $G_0$ be the Green operator in the
  elliptic complex $C(A_0)$. Since $u_s\in V_0^\perp$,
  \begin{equation*}
    u_s=\bar{\partial}_{A_0}^*\bar{\partial}_{A_0}G_0u_s=\bar{\partial}_{A_0}^*G_0\bar{\partial}_{A_0}u_s=\bar{\partial}_{A_0}^*G_0(-\bar{\partial}_{A_0}s)=\bar{\partial}_{A_0}^*G_0(as)
  \end{equation*}
  Hence, $\pi_A$ has a natural extension
  \begin{equation*}
    \tilde{\pi}_A\colon\Omega^0(E)\to\Omega^0(E)\qquad s\mapsto s+\bar{\partial}_{A_0}^*G_0(as)
  \end{equation*}
  satisfying the following estimate
  \begin{equation*}
    \|\bar{\partial}_{A_0}^*G_0(as)\|_{L_k^2}\leq
    C\|as\|_{L_{k-1}^2}\leq C\|as\|_{L_k^2}\leq
    C\|a\|_{L_k^2}\|s\|_{L_k^2}\leq C\|a\|_{C^\infty}\|s\|_{L_k^2}
  \end{equation*}
  where we have used the Sobolev multiplication theorem (see \cite[Theorem
  4.4.1]{friedman1998gauge}). Therefore, if $A_1,A_2\in\B^{ss}$ and
  $\bar{\partial}_{A_i}=\bar{\partial}_{A_0}+a_i$ for some
  $a_i\in\Omega^{0,1}(E)$, we have
  \begin{equation}
    \label{eq:pi-bounded}
    \|(\tilde{\pi}_{A_2}-\tilde{\pi}_{A_1})s\|_{L_k^2}=\|\bar{\partial}_{A_0}^*G_0(a_2-a_1)s\|_{L_k^2}\leq C\|a_2-a_1\|_{C^\infty}\|s\|_{L_k^2}
  \end{equation}
  Now if $U$ is sufficiently small, we may assume that
  \begin{equation*}
    \|\bar{\partial}_{A_0}^*G_0(as)\|_{L_k^2}\leq(1/2)\|s\|_{L_k^2}
  \end{equation*}
  so that
  \begin{equation}
    \label{eq:pi-injective}
    \|\tilde{\pi}_As\|_{L_k^2}\geq(1/2)\|s\|_{L_k^2}
  \end{equation}
  This shows that $\tilde{\pi}_A$ is injective. Since $H^1(\E_A)=0$,
  $\dim V_A=\dim V_0=p$, $\pi_A$ is an isomorphism. Therefore, the map
  \begin{equation*}
    X\times V_0\xrightarrow{1\times\pi^{-1}_A}X\times
    V_A\xrightarrow{(x,s)\mapsto s(x)}\E_A
  \end{equation*}
  realizes $\E_A$ as a quotient of $V_0\otimes\sO_X\cong\sO_X^p$,
  since $\E_A$ is generated by global sections. As a consequence, the
  classifying map
  \begin{equation*}
    \sigma\colon U\to F^{ss}\qquad (A,\Phi)\mapsto(\sO_X^p\to\E_A\to0,\Phi)
  \end{equation*}
  is well-defined.

  Now, we show that $\sigma$ is continuous. Let $G(p,r)$ be the
  Grassmannian parameterizing isomorphism classes of quotients
  $\C^p\to V\to0$, where $V$ is a vector space of dimension $r$. Over
  $G(p,r)$, there is a universal quotient bundle $H\to G(p,r)$. Fix $x\in X$ and
  choose a basis for the fiber $(\K_X)_x$ of the canonical bundle
  $\K_X$ over $x$. Therefore, any Higgs field
  $\Phi\in H^0(\End\E\otimes\K_X)$ induces an endomorphism
  $\Phi_x\colon E_x\to E_x\otimes(\K_X)_x\cong E_x$. Then, Nitsure
  showed in \cite{nitsure1991moduli} that there is a morphism
  \begin{equation*}
    \tau_x\colon F\to\End H\qquad
    (\sO_X^p\to\E_A\to0,\Phi)\mapsto(\C^p\to E_x,\Phi_x\colon E_x\to E_x)
  \end{equation*}
  where $\C^p\to E_x$ is obtained by evaluating the map
  $\sO_X^p\to\E_A$ at $x$. Moreover, \cite[Proposition
  5.7]{nitsure1991moduli} states that there are $N$ points
  $x_1,\cdots, x_N\in X$ such that $\{\tau_{x_i}\}$ induces an
  injective and proper morphism (in the category of schemes)
  $ \tau\colon F^{ss}\to W $ for some open subset $W$ of $(\End
  H)^N$. Therefore, the underlying continuous map of $\tau$ is a
  closed embedding with respect to the analytic topology. Hence,
  $\sigma$ is continuous if the composition
  \begin{equation*}
    \sigma_x\colon U\xrightarrow{\sigma} F^{ss}\xrightarrow{\tau_x}\End H
  \end{equation*}
  is continuous for any $x\in X$. More explicitly, $\sigma_x$ is given
  by
  \begin{equation*}
    (A,\Phi)\mapsto(V_0\to E_x\to0,\Phi_x\colon E_x\to E_x)
  \end{equation*}
  where $V_0\to E_x$ is defined by
  \begin{equation*}
    V_0\xrightarrow{\pi_A^{-1}}V_A\xrightarrow{s\mapsto s(x)} E_x
  \end{equation*}
  Clearly, the map $\Phi\mapsto\Phi_x$ is continuous. It suffices to
  show that
  \begin{equation*}
    A\mapsto (V_0\to E_x\to0)
  \end{equation*}
  is continuous. Fix $s\in V_0$ and $A_1,A_2\in U$. Write
  $\bar{\partial}_{A_i}=\bar{\partial}_{A_0}+a_i$ for some
  $a_i\in\Omega^{0,1}(E)$ ($i=1,2$). Then, the following estimate
  follows from \eqref{eq:pi-bounded}, \eqref{eq:pi-injective}, and
  Sobolev embedding $L_k^2\hookrightarrow C^0$,
  \begin{align*}
    |(\pi_{A_1}^{-1}-\pi_{A_2}^{-1})s(x)|
    &\leq\|(\pi_{A_1}^{-1}-\pi_{A_2}^{-1})s\|_{C^0}\\
    &\leq C\|(\pi_{A_1}^{-1}-\pi_{A_2}^{-1})s\|_{L_k^2}\\
    &\leq
      C\|\tilde{\pi}_{A_1}^{-1}(s-\tilde{\pi}_{A_1}\tilde{\pi}_{A_2}^{-1}s)\|_{L_k^2}\\
    &\leq C\|s-\tilde{\pi}_{A_1}\tilde{\pi}_{A_2}^{-1}s\|_{L_k^2}\\
    &=C\|(\tilde{\pi}_{A_2}-\tilde{\pi}_{A_1})\pi_{A_2}^{-1}s\|_{L_k^2}\\
    &\leq C\|a_2-a_1\|_{C^\infty}\|\tilde{\pi}_{A_2}^{-1}s\|_{L_k^2}\\
    &\leq C\|a_2-a_1\|_{C^\infty}\|s\|_{L_k^2}
  \end{align*}
  Hence, $A\mapsto(V_0\to E_x\to0)$ is continuous.
\end{proof}

\subsection{Holomorphicity}
We continue to show that the comparison map $i$ is a
biholomorphism. Let $\M_{an}^s$ and $\M_{alg}^s$ be the subset of
$\M_{an}^s$ and $\M_{alg}^s$ consisting of stable Higgs bundles. We
first show that the restriction $i\colon\M_{an}^s\to\M_{alg}^s$ is a
biholomorphism. By \cite[Theorem 4.7]{Simpson1994}, $\M_{alg}^s$ is
open in $\M_{alg}$. By \cite[Corollary 11.7]{Simpson1994} and \cite[Proposition
7.1]{nitsure1991moduli}, we see that $\M_{alg}^s$ is smooth. On the
other hand, a polystable Higgs bundle $(A,\Phi)$ is stable if and only
if its $\G^\C$-stabilizer is equal to $\C^*$ or equivalently
$\dim\bH^0(C_{\mu_\C}(A,\Phi))=1$. Since $\C^*$ is contained in every
$\G^\C$-stabilizer, by the upper semicontinuity of dimensions of
cohomology (see \cite[Chapter VII, (2.37)]{Kobayashi2014}), we
conclude that $\M_{an}^s$ is open in $\M_{an}$.
\begin{proposition}\label{sec:holomorphicity-smooth-Ms}
  $\M_{an}^s$ is a smooth submanifold of $\M_{an}$
\end{proposition}
\begin{proof}
  Fix $(A,\Phi)\in\B^s$ that satisfies Hitchin's equation. Let $K$ be
  its $\G$-stabilizer so that $K^\C$ is its $\G^\C$-stabilizer. To
  show that $\M_{an}^s$ is smooth, we will use
  Theorem~\ref{sec:introduction-local-model}. It is enough to show
  that $\nu_{0,\C}^{-1}(0)\sslash K^\C=\bH^1$. In fact, since
  $K^\C=\C^*$, $K^\C$ acts on $\bH^1$ trivially. Moreover,
  $\nu_{0,\C}(x)=\frac{1}{2}H[x,x]$ is trace-free for every
  $x\in \bH^1$. Since $\bH^2(C_{\mu_\C})=\C^*\omega_X$, we conclude
  that $H[x,x]=0$ for every $x\in\bH^1$, where $\omega_X$ is a fixed
  K\"ahler form on $X$.
\end{proof}

Fix $[A,\Phi]\in\M_{an}^s$ such that $(A,\Phi)\in\B^s$ satisfies
Hitchin's equation. By Corollary
\ref{sec:kuran-local-models-Cinfty-Sobolev-homeo} and
Proposition~\ref{sec:holomorphicity-smooth-Ms}, we see that
$\varphi\colon\Z\to\M_{an}^s$ is a biholomorphism onto an open
neighborhood of $[A,\Phi]$ in $\M_{an}^s$, where $\Z$ is an open
neighborhood of $0$ in $\bH^1$ and $\varphi$ the map induced by the
Kuranishi map $\theta\colon\Z\to\B^s$ (see
Section~\ref{sec:kuran-local-models}). Therefore, to show that
$i|_{\M_{an}^s}$ is holomorphic, it is enough to show that
$i\varphi\colon\Z\to\M_{alg}^s$ is holomorphic. By the remark after
the proof of \cite[Corollary 5.6]{Simpson1994}, we see that the
analytification of $\M_{alg}$ is the coarse moduli space of semistable
Higgs bundles in the category of complex spaces. Therefore, to show
that $i\varphi$ is holomorphic, we need to construct a family
$(\V,\mathbf{\Phi})$, call the \textit{Kuranishi family} associated
with $\theta$, of stable Higgs bundles over $\Z$ such that
$(\V_t,\mathbf{\Phi}_t)$ is isomorphic to $(\E_{A_t},\Phi_t)$ for
every $t\in\Z$, where $(A_t,\Phi_t)=\theta(t)$. In general, a family
$(\V,\mathbf{\Phi})$ of Higgs bundles over a complex space $T$ is a
holomorphic vector bundle $\V\to X\times T$ together with a
holomorphic section
$\mathbf{\Phi}\in H^0(X\times T,p_X^*\K_X\otimes\End\V)$, where
$p_X\colon X\times T\to X$ is the projection onto the first factor.
\begin{proposition}\label{sec:holomorphicity-Kuranishi-family}
  For any $(A,\Phi)\in\B^s$, let $\theta\colon\Z\to\B^s$ be the
  Kuranishi map defined by $(A,\Phi)$. Then, there exists a Kuranishi
  family $(\V,\mathbf{\Phi})$ of stable Higgs bundles over $\Z$ such
  that $(\V_t,\mathbf{\Phi}_t)$ is isomorphic to $(\E_{A_t},\Phi_t)$
  for every $t\in\Z$, where $(A_t,\Phi_t)=\theta(t)$.
\end{proposition}
\begin{proof}
  We adapt the proof of \cite[Proposition
  2.6]{friedman2013smooth}. Let $V=p_X^*E$ be the smooth vector bundle
  over $X\times\Z$, and $\mathbf{\Phi}(x,t):=\Phi_t(x)$ can be
  regarded as a smooth section of
  $p_X^*\Lambda^{1,0}X\otimes\End(U)\subset\Omega^{1,0}(X\times\Z,\End
  U)$. Then, we need to put a holomorphic structure on $V$ so that
  $\mathbf{\Phi}$ is a holomorphic section.
  
  Let $\{s_i\}$ be a smooth local frame for $E$. Then $\{p_X^*s_i\}$
  is a smooth local frame for $V$. Then, we define a
  $\bar{\partial}$-operator
  $\bar{\partial}_V\colon\Omega^0(V)\to\Omega^{0,1}(V)$ by the
  requirement that
  \begin{equation*}
    \bar{\partial}_V(p_X^*s_i)=\bar{\partial}_{A_t}s_i
  \end{equation*}
  Here, $\bar{\partial}_{A_t}s_i$ is regarded as a local section of
  $\Lambda^{0,1}(X\times\Z)\otimes V$. It is easy to show that
  $\bar{\partial}_V$ is independent of the choices of smooth local
  frames $\{s_i\}$. Therefore, $\bar{\partial}_V$ is a well-defined
  $\bar{\partial}$-operator on $V$.

  Then, we show that $\bar{\partial}_V$ is integrable so that
  $\V=(V,\bar{\partial}_V)$ is a holomorphic vector bundle over
  $X\times\Z$. Write $\bar{\partial}_{A_t}s_i=f_i^js_j$ for some
  smooth local function $f_i^j$ on $X\times\Z$. Since $\theta$ is
  holomorphic, each $f_i^j$ is holomorphic in the direction of
  $\Z$. As a consequence,
  \begin{equation*}
    \bar{\partial}_V^2(p_X^*s_i)=\bar{\partial}_{X\times\Z}f_i^j\wedge
    s_j+f_i^j\bar{\partial}_{A_t}s_j=\bar{\partial}_Xf_i^j\wedge s_j+f_i^j\bar{\partial}_{A_t}s_j
  \end{equation*}
  where $\bar{\partial}_{X\times\Z}$ and $\bar{\partial}_X$ are usual
  $\bar{\partial}$-operator on complex manifolds $X\times\Z$ and $X$,
  respectively. On the other hand,
  \begin{equation*}
    0=\bar{\partial}_{A_t}^2s_i=\bar{\partial}_Xf_i^j\wedge s_j+f_i^j\bar{\partial}_{A_t}s_j
  \end{equation*}

  Then, we show that $\bar{\partial}_V\mathbf{\Phi}=0$. Write
  $\Phi_s=\phi^is_i$ for some smooth local function $\phi^i$ on
  $X\times\Z$. Since $\theta$ is holomorphic, $\phi^i$ is holomorphic
  in the direction of $\Z$. As a consequence,
  \begin{equation*}
    \bar{\partial}_V\mathbf{\Phi}=\bar{\partial}_{X\times\Z}\phi^i\wedge
    s_i+\phi^i\bar{\partial}_{A_t}s_i=\bar{\partial}_X\phi^i\wedge s_i+\phi^i\bar{\partial}_{A_t}s_i=\bar{\partial}_{A_t}\Phi_t=0
  \end{equation*}

  Finally, we need to show that if $(\V_t,\mathbf{\Phi}_t)$ is
  isomorphic to $(\E_{A_t},\Phi_t)$ for any $t\in\Z$. If
  $i_t(x)=(x,t)$ is the holomorphic map $X\to X\times\Z$, then the
  holomorphic structure on $i_t^*\V$ is given by the pullback
  $\bar{\partial}$-operator $i_t^*\bar{\partial}_V$. Since
  \begin{equation*}
    [i_t^*(\bar{\partial}_V)](i_t^*p_X^*s)=i_t^*(\bar{\partial}_Vs)=\bar{\partial}_{A_t}s
  \end{equation*}
  for any smooth local section $s$ of $E$, we see that $i_t^*\V$ is
  isomorphic to $\E_{A_t}$. Moreover, $i_t^*\mathbf{\Phi}=\Phi_t=\Phi$.
\end{proof}
\begin{corollary}\label{sec:holomorphicity-biholo-Ms}
  The comparison map $i\colon\M_{an}^s\to\M_{alg}^s$ is a biholomorphism.
\end{corollary}
\begin{proof}
  Since the analytification of $\M_{alg}$ is the coarse moduli space
  of semistable Higgs bundles in the category of complex spaces, the
  family $(\V,\mathbf{\Phi})$ constructed in
  Proposition~\ref{sec:holomorphicity-Kuranishi-family} induces a
  holomorphic map
  \begin{equation*}
    \Z\to\M_{alg}^s\qquad t\mapsto[\V_t,\mathbf{\Phi}_t]
  \end{equation*}
  On the other hand, the map $i\varphi\colon\Z\to\M_{alg}^s$ is given by
  \begin{equation*}
    i\varphi(t)=i[A_t,\Phi_t]=[\E_{A_t},\Phi_t]=[\V_t,\mathbf{\Phi}_t]
  \end{equation*}
  Hence, $i\varphi$ is holomorphic. Since both $\M_{an}^s$ and
  $\M_{alg}^s$ are smooth complex manifolds, and $i$ is a holomorphic
  bijection, $i$ is a biholomorphism.
\end{proof}
Then, we extend the holomorphicity of $i^{-1}$ on $\M_{alg}^s$ to the
full moduli space $\M_{alg}$.
\begin{corollary}
  The map $i^{-1}\colon\M_{alg}\to\M_{an}$ is holomorphic.
\end{corollary}
\begin{proof}
  Recall that $\M_{an}$ is assumed to be reduced, and $\M_{alg}$ is
  reduced. Take a holomorphic $f\colon U\to\C$ where $U$ is an open
  subset of $\M_{an}$. Then, the pullback $(i^{-1})^*f$ is continuous
  on the open set $i(U)$ and holomorphic on $i(U)\cap\M_{alg}^s$. By
  \cite{kuhlmann1961normalisierung}, the normality of $\M_{alg}$
  implies the normality of its analytification. Since $\M_{alg}^s$ is
  open in the Zariski topology, $\M_{alg}\setminus\M_{alg}^s$ is a
  closed analytic subset of $\M_{alg}$ in the analytic topology. Since
  $\M_{alg}\setminus\M_{alg}^s$ has codimension $\geq2$ (see
  \cite[Theorem II.6]{faltings1993stable}), then the Riemann extension
  theorem for normal complex spaces implies that the restriction
  $(i^{-1})^*f\colon\M_{alg}^s\cap i(U)\to\C$ can be extended to a
  holomorphic function $g$ on $i(U)$. Since $\M_{alg}$ is irreducible,
  the open set $\M_{alg}^s$ is dense in the Zariski topology and hence
  in the analytic topology (\cite[\S10, Theorem
  1]{mumford1996red}). Since both $(i^{-1})^*f$ and $g$ are continuous
  and agree on an open dense subset $\M_{alg}^s\cap i(U)$ of $i(U)$,
  $(i^{-1})^*f=g$. This shows that $i^{-1}$ is holomorphic.
\end{proof}

The final ingredient is the normality of $\M_{an}$.
\begin{lemma}\label{sec:holomorphicity-normality}
  $\M_{an}$ is a normal complex space.
\end{lemma}
\begin{proof}
  Let us temporarily use $Q$ to mean $\nu_{0,\C}^{-1}(0)$ viewed as an
  affine variety in $\bH^1$ and $Q^{an}$ to mean the analytification
  of $Q$. By Theorem~\ref{sec:sing-kuran-spac-local-model}, it
  suffices to prove that $Q^{an}\sslash K^\C$ is normal at the origin
  $[0]$. Here, $Q^{an}\sslash K^\C$ is the analytic GIT quotient of
  $Q^{an}$ by $K^\C$. By \cite{heinzner1994reduction}, the
  analytification of the affine GIT quotient $Q\sslash K^\C$ is
  $Q^{an}\sslash K^\C$.

  Now, we fix a Higgs bundle $(A,\Phi)$ such that $\mu(A,\Phi)=0$. By
  choosing a point $x\in X$, the holomorphic bundle $(\E_A,\Phi,x)$
  defines a point in the moduli space $\mathbf{R}_{Dol}(X,x,n)$ of the
  semistable Higgs bundles of rank $n$ and degree $0$ and with a frame
  at $x$. In \cite[Corollary 11.7]{Simpson1994}, it is shown that
  $\mathbf{R}_{Dol}(X,x,n)$ is normal. Moreover, in the proof of
  \cite[Proposition 10.5]{Simpson1994}, it is shown that the formal
  completion of $Q$ (regarded as an affine variety in $\bH^1$) at $0$
  is isomorphic to the formal completion of a subscheme $Y$ at
  $(\E_A,\Phi,x)$. Here, $Y$ is a local slice, provided by Luna's
  slice theorem (see \cite[Theorem 4.2.12]{huybrechts2010geometry}) at
  $(\E_A,\Phi,x)$ for the $GL_n(\C)$ action on
  $\mathbf{R}_{Dol}(X,x,n)$. Moreover, since $\mathbf{R}_{Dol}(X,x,n)$
  is normal at $(\E_A,\Phi,x)$, $Y$ can be taken to be normal at
  $(\E_A,\Phi,x)$. As a consequence, the formal completion of $Q$ is
  normal at $0$. By
  \cite[\href{https://stacks.math.columbia.edu/tag/0FIZ}{Tag
    0FIZ}]{stacks-project}, $Q$ is normal at $0$. Since taking
  invariants commutes with localizations and preserves the normality,
  we conclude that $Q\sslash K^\C$ is normal at $[0]$. Since normality
  is preserved by the analytification (see
  \cite{kuhlmann1961normalisierung}), we see that $Q^{an}\sslash K^\C$
  is normal at $[0]$.
\end{proof}

The proof of Theorem~\ref{sec:introduction-comparison} rests on the
following theorem.
\begin{theorem}[{\cite[Theorem, p.166]{grauert2012coherent}}]\label{sec:holomorphicity-thm-grauert}
  Let $f\colon X\to Y$ be an injective holomorphic map between reduced
  and pure dimensional complex spaces. Assume that $Y$ is normal and
  that $\dim X=\dim Y$. Then $f$ is open, and $f$ maps $X$
  biholomorphically onto $f(X)$. In particular, the space $X$ is
  normal.
\end{theorem}

\begin{proof}[Proof of Theorem~\ref{sec:introduction-comparison}]
  Now the map
  \begin{equation*}
    i^{-1}\colon\M_{alg}\to\M_{an}
  \end{equation*}
  is a holomorphic homeomorphism. To use
  Theorem~\ref{sec:holomorphicity-thm-grauert}, we verify that
  $\M_{an}$ is pure dimensional, normal and
  $\dim\M_{an}=\dim\M_{alg}$. By
  Lemma~\ref{sec:holomorphicity-normality}, $\M_{an}$ is normal. Since
  $\M_{alg}$ is connected in the analytic topology, $\M_{an}$ is
  connected. Then, the normality and connectedness of $\M_{an}$
  implies that $\M_{an}$ is irreducible and hence pure dimensional
  (see \cite[Theorem, p.168]{grauert2012coherent}). Finally, by
  Corollary~\ref{sec:holomorphicity-biholo-Ms},
  $\dim\M_{an}=\dim\M_{alg}$.
\end{proof}

\bibliography{..//..//references}
\bibliographystyle{abbrv}
\end{document}